\documentclass{amsart}
\usepackage{systeme}
\usepackage{fullpage}
\usepackage{fontenc}
\usepackage{lmodern}
\usepackage{amsfonts, amsmath, amssymb, amscd, amsthm}
\usepackage{color}
\usepackage{subfig}
\usepackage{rotate,float}
\usepackage{fancybox,graphics,rotating}
\newcommand{\R}{\mathbb R}
\def\supp{\mathop{\rm supp}\nolimits}
\def\H{{\mathcal H}}
\def\Z{{\mathbb Z}}
\def\M{{\mathcal M}}
\newcommand{\N}{\mathbb N}

\newcommand{\dsp}{\displaystyle}
\newcommand{\eps}{\varepsilon}
\newtheorem{proposition}{Proposition}
\newtheorem{theorem}{Theorem}

\newtheorem{corollary}{Corollary}
\newtheorem{lemma}{Lemma}

\newtheorem{hypothesis}{Hypothesis}
\newcommand{\weaktendsto}[1]{\renewcommand{\arraystretch}{0.5}
\begin{array}[t]{c}
\rightharpoonup \\
{ \scriptstyle #1 }
\end{array}
\renewcommand{\arraystretch}{1}}

\title{\Large{\bf{On the orbital stability of  the Degasperis-Procesi antipeakon-peakon profile}}}
\date{\today}
\begin{document}
\author[B. Khorbatly and L. Molinet]{Bashar Khorbatly and Luc Molinet}
\address{Bashar Khorbatly, Laboratory of Mathematics-EDST, Department of Mathematics, Faculty of Sciences 1, Lebanese University, Beirut, Lebanon  and \\
 Institut Denis Poisson, Universit\'e de Tours, Universit\'e d'Orl\'eans, CNRS,  Parc Grandmont, 37200 Tours, France.}
\email{Bashar-elkhorbatly@hotmail.com}
\address{Luc Molinet, Institut Denis Poisson, Universit\'e de Tours, Universit\'e d'Orl\'eans, CNRS,  Parc Grandmont, 37200 Tours, France.}
\email{Luc.Molinet@univ-tours.fr}

\maketitle
\begin{abstract}
In this paper, we prove an orbital stability result for the Degasperis-Procesi peakon with respect to perturbations having a momentum density that is first negative and then positive. This leads to  the orbital  stability of  the antipeakon-peakon profile with respect to such perturbations.
\end{abstract}

\section{Introduction}

In this paper, we consider the Degasperis-Procesi equation (DP) first derived in \cite{DP}, usually written as
\begin{equation}\label{DP1}
\left\{ 
\begin{array}{l}
u_{t}-u_{txx}+4uu_{x}=3u_{x}u_{xx}+uu_{xxx},
\qquad(t,x)\in\mathbb{R}_{+}\times\mathbb{R},\\
u(0,x)=u_0(x) , \quad x\in\R \; .
\end{array}
\right. 
\end{equation}
The DP equation has been proved to be physically relevant for water waves (see \cite{MR2481064}) as an asymptotic shallow-water approximation to the Euler equations in some specific regime. It shares a lot of properties with the famous Camassa-Holm equation (CH) that reads
 \begin{equation}
u_t -u_{txx}=- 3 u u_x + 2 u_x u_{xx} + u u_{xxx}, \quad 
(t,x)\in\R_+\times \R \, . \label{Camassa1}\\
\end{equation}
In particular, it has a bi-hamiltonian structure, it is completely integrable (see \cite{DHH})  and has got the same explicit peaked solitary waves. These solitary waves are called \textit{peakons} whenever $c>0$  and {\it antipeakons} whenever $c<0$ and are defined by
\begin{equation}
u(t,x)=\varphi_{c}(x-ct)=c\varphi(x-ct)
=ce^{-|x-ct|},
\quad c\in\mathbb{R}^{*},
\quad(t,x)\in\mathbb{R}^2\; .
\label{1.7}
\end{equation}
Note that to give a sense to these solutions one has to apply 
 $(1-\partial^{2}_{x})^{-1}$ to \eqref{DP1}, to rewrite it under the form
\begin{equation}\label{DP}
u_{t}+\frac{1}{2}\partial_{x}( u^{2})+\frac{3}{2}(1-\partial^{2}_{x})^{-1}\partial_{x}(u^{2})=0,
\qquad(t,x)\in\mathbb{R}_{+}\times\mathbb{R}.
\end{equation}
However,  in  contrast with the CH equation, the DP equation has also shock peaked waves (see for instance \cite{L})
  which are given by 
  $$
  u(t,x) =-\frac{1}{t+k} \text{sgn}(x) e^{-|x|} , \quad k>0 \quad (t,x)\in \R_+\times \R \; .
  $$
  Another important difference between the CH and the DP equations  is due to the fact that  the DP conservations laws permit only to control the $ L^2$-norm of the solution whereas the $ H^1$-norm is a conserved quantity for the CH equation.
   In particular, without any supplementary hypotheses, the solutions of the DP equation may be unbounded contrary to the CH-solutions.  
In this paper we will make use of the three following conservation laws of the DP equation :
\begin{align}
M(u)=\int_{\R} y , \quad 
E(u)=\int_{\mathbb{R}}yv=\int_{\mathbb{R}}\left(4v^{2}+5v^{2}_{x}+v^{2}_{xx}\right) \label{EE}\\
\text{and} \quad \quad F(u)=\int_{\mathbb{R}}u^{3}=\int_{\mathbb{R}}\left(-v^{3}_{xx}+12vv^{2}_{xx}-48v^{2}v_{xx}+64v^{3}\right),
\label{i}
\end{align}
where  $y=(1-\partial^{2}_{x})u\text{ and } v=(4-\partial^{2}_{x})^{-1}u$. \\
It is worth noticing that these two variables, the momentum density $y=(1-\partial^{2}_{x})u$  and  the smooth variable $v=(4-\partial^{2}_{x})^{-1}u$ play a crucial role in the DP dynamic. In the sequel we will  often make use of the fact that 
\eqref{DP1} can be rewritten under the form 
\begin{equation}\label{DP2}
y_t+uy_x+3u_xy=0,\qquad(t,x)\in\mathbb{R}_{+}\times\mathbb{R},
\end{equation}
which is a transport equations for the momentum density as well as under the form 
\begin{equation}\label{DP3}
v_t=-\partial_x(1-\partial_x^2)^{-1} u^2 ,\qquad(t,x)\in\mathbb{R}_{+}\times\mathbb{R}.
\end{equation}
Note that, in the same way as $ v $ is associated with $ u $,  we will associate with the peakon profile  $\varphi_c$ the so-called    {\it smooth-peakon} profile $ \rho_c $ that is given by
\begin{equation}\label{1.88}
\rho_{c}=(4-\partial^{2}_{x})^{-1}\varphi_{c}=\frac{1}{4}e^{-2\vert \cdot\vert}\ast\varphi_c
=\frac{c}{3}e^{-|\cdot|}-\frac{c}{6}e^{-2|\cdot|}\ge 0
\; .
\end{equation}
\begin{figure}[!htb]
\vspace{-0.5cm}
\centering
\subfloat[
 Peakon and antipeakon profiles]
{\includegraphics[width=8cm, height=6.5cm]{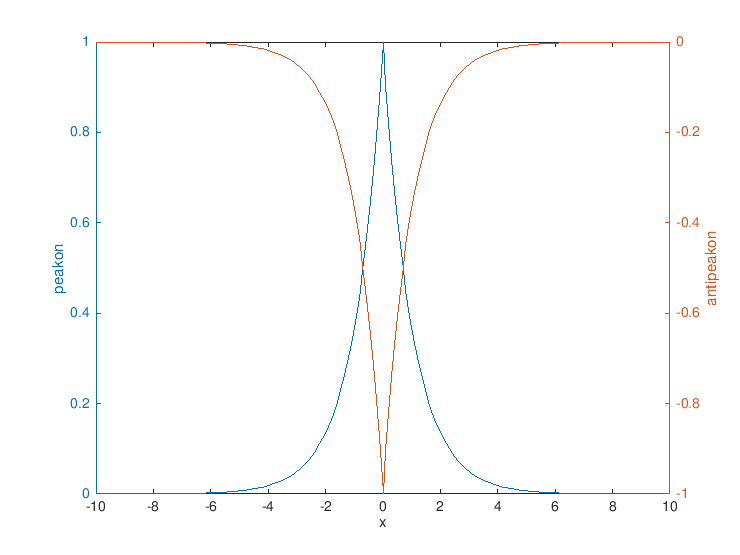}\label{F2}}
\subfloat[
Smooth peakon and    
 smooth antipeakon profiles]
{\includegraphics[width=8cm, height=6.5cm]{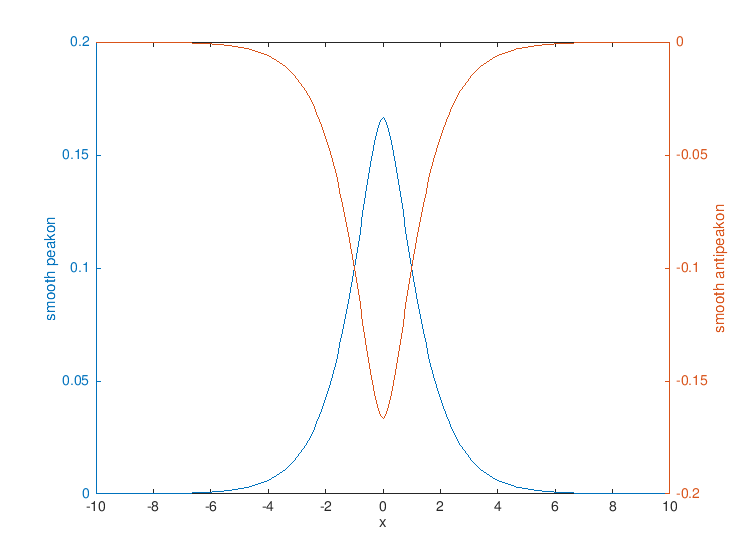}}
\caption{(A) Peakon and antipeakon representative curves with speed  $c=\pm 1$. They are even functions that admit a single maximum $c $ (resp. maximum $-c$)  at the origin. (B) Smooth  peakon and smooth antipeakon representative curves with speed  $c=\pm 1$. They are even $ C^2$ functions that admit a single maximum $c/6$ (resp. minimum $-c/6$)  at the origin. }
\label{smothanipek}
\end{figure}
In \cite{LL} (see also \cite{AK} for a great simplification)  an orbital stability\footnote{See \cite{LM1} for an asymptotic stability result in the class of functions with a  positive momentum density.} result is proven for the DP peakon  by adapting the approach first developed by Constantin and Strauss \cite{CS} for the Camassa-Holm peakon. However, in deep contrast to the Camassa-Holm case, the proof in \cite{LL} (and also in  \cite{AK}) crucially use that the  momentum density of the perturbation is non negative. This is absolutely required for instance in [\cite{LL}, Lemma 3.5] to get the crucial estimate on the auxiliary function $ h$ (see Section   \ref{sect5}  for the definition of $ h$)). Up to our knowledge, there is no available stability result for the Degasperis-Procesi peakons without this requirement on the momentum density 
 and one of the main contribution of  this work is to give a first stability result for the DP peakon with respect to  perturbations that do not share this sign requirement.  At this stage, it is worth noticing that the global existence of smooth solutions to the DP equation is only known for initial data that have  either   a momentum density with a constant sign or a momentum density that is first non negative and then non positive.  
 
The first part of this  paper is devoted to the proof of a stability result for  the peakon with respect to perturbations that belong to this second class of initial data. We would like to emphasize that the key supplementary argument with respect to the case of a non negative momentum density is of a dynamic nature. Inspired by similar considerations for the Camassa-Holm equation contained in \cite{LM1}, we study the dynamic of the momentum density $ y(t) $ at the left of a 
smooth curve $ x(t)$ such that $u(t,\cdot)-\varphi_c(\cdot-x(t)) $ remains small for  all $ t\in [0,T]$ with $ T >0$ large enough. This is in deep contrast with the arguments in the case $ y\ge 0 $ and with the common arguments for orbital stability
 that are of static nature : They only use the   conservation laws together with the continuity of the solution. 
 
In a second time, we combine this stability result with some almost  monotony results to get the orbital stability of the DP antipeakon-peakon profile and more generally of  trains of antipeakon-peakons.

Before stating our results let us introduce some notations and some function spaces that will appear in the statements. 
For $ p\in [1,+\infty] $ we denote by $ L^p(\R) $ the usual Lebesgue spaces endowed with their usual norm
 $ \| \cdot\|_{L^p} $.  We  notice that by integration by parts, it holds
$$
\|u(t,\cdot)\|^{2}_{L^{2}}=\int_{\mathbb{R}}(4v-v_{xx})^{2}dx=\int_{\mathbb{R}}\left(16v^{2}+8v^{2}_{x}+v^{2}_{xx}\right)dx
$$
and thus 
$$
E(u) \le \|u\|_{L^2}^2\le 4 E(u)\; .
$$
Therefore, 
$E(\cdot)$ is equivalent to $\|\cdot\|^{2}_{L^{2}(\mathbb{R})}$ and in the sequel 
 of this paper we set 
\begin{equation}
\|u\|_{\mathcal{H}}=\sqrt{E(u)}\quad  \text{ so that }\quad \|u\|_{\mathcal{H}}\le \|u\|_{L^2} \le 2  \|u\|_{\mathcal{H}}
\label{1.5}
\end{equation}
As in \cite{CM}, we will work in the space $ Y $ defined by 
\begin{equation}\label{defy}
Y:=\Bigl\{u\in L^2(\R) \quad\text{with} \quad \quad u-u_{xx}\in \M(\R)  \Bigr\}
\end{equation}
where $ \M(\R) $ is the space of finite Radon measure on $ \R$ that is endowed with the norm
 $  \|\cdot\|_{\M} $ where 
$$
\|y\|_{\M}:=\sup_{\varphi\in C(\R), \|\varphi\|_{L^\infty}\le 1} |\langle y, \varphi \rangle | \;.
$$
\begin{hypothesis}\label{hyp}
We will say that $u_0\in Y$ satisfies Hypothesis \ref{hyp} if there exists $x_0\in\R$ such that its momentum density $y_0=u_0-u_{0,xx}$ satisfies
\begin{equation}
\text{supp }y_0^{-}\subset]-\infty,x_0]\qquad\text{ and }\qquad\text{supp }y_0^{+}\subset[x_0,+\infty[.
\end{equation}
where $y_0^+$ and $ y_0^- $ are respectively the positive and the negative part of the Radon measure $ y_0$.
\end{hypothesis}
\begin{theorem}[Stability of a single  Peakon]\label{stabpeakon}
There exists $ 0<\eps_0<1 $ such that for any $ c>0 $, $A>0 $
 and $ 0<\eps<\eps_0 \frac{1\wedge c^2}{(2+c)^3} $,  there exists $ 0<\delta=\delta(A,\eps,c)\le \varepsilon^4 $ such that   for  any  $u_0 \in Y $ satisfying Hypothesis \ref{hyp} with
 \begin{equation}\label{P2}
\Vert u_0-\varphi_{c}\Vert_{\mathcal{H}}\le \delta\le \eps^4
\end{equation}
and
 \begin{equation}
\Vert u_0-u_{0,xx}\Vert_{\mathcal{M}}\leq A,
\label{P1}
\end{equation}
the emanating solution of the DP equation satisfies
\begin{equation}
\|u(t,\cdot)-\varphi_{c}(\cdot-\xi(t))\|_{\mathcal{H}}\le 2 (2+c) \; \eps,~~\forall t\in\R_+
\label{P3}
\end{equation}
and 
\begin{equation}
\|u(t,\cdot)-\varphi_{c}(\cdot-\xi(t))\|_{L^\infty}\le 8 (2+c)^{2}\varepsilon^{2/3},~~\forall t\in \R_+\; , 
\label{P44}
\end{equation}
where $\xi(t)\in\mathbb{R}$ is the only point where the function 
$v(t,\cdot)=(4-\partial^{2}_{x})^{-1}u(t,\cdot)$ reaches its maximum on $\R$.
\end{theorem}
Combining the above stability of a single peakon with the general framework first introduced in \cite{MMT} and more precisely following  \cite{MR2542735}-\cite{dikalm} we obtain the stability of a train of well-ordered antipeakons and peakons. This contains in particular the stability of the antipeakon-peakon profile. 
\begin{theorem}\label{stabantipeakonpeakon}
Let be given $ N_- \in \N^*$ negative   velocities $ c_{-N_-} <..<c_{-1}<0 $, $ N_+ \in \N^*$ positive velocities $0<c_1<..<c_{N_+} $ and $ A>0 $. There exist $B=B(\vec{c})>0$, $L_0=L_0(A,\vec{c}) >0 $ and $ 0<\eps_0=\varepsilon_0(\vec{c})<1 $ such that for any $ 
0<\varepsilon<\eps_0(\vec{c}) $ there exists $ 0<\delta(\eps,A,\vec{c})<\eps^4	$  such that if  $ u\in C(\R_+;H^1) $ is
   the  solution of the DP equation emanating from $ u_0\in Y $, satisfying Hypothesis \ref{hyp} with 
    \begin{equation}
\Vert u_0-u_{0,xx}\Vert_{\mathcal{M}}\leq A,
\label{huhu0}
\end{equation}
and
   \begin{equation}\label{huhu}
 \|u_0-\sum_{j=-N_-\atop j\neq 0}^{N_+}  \varphi_{c_j}(\cdot-z_j^0) \|_{\H} \le \delta\le \eps^4
 \end{equation}
 for some  $ z^0_{-N_-}<..<z^0_{-1}<z^0_1<\cdot\cdot\cdot<z^0_{N_+} $ such that 
 \begin{equation} \label{distz}
  z_j^0-z_{q}^0\ge L\ge L_0 ,   \quad \forall (j,q) \in \Bigl( [[-N_- , N_+]]\setminus\{0\}\Bigr)^2, \quad j> q  \; , 
 \end{equation}
  then there exist $ N_-+N_+ $  functions $ \xi_{-N_-}(\cdot), .., \xi_{-1}(\cdot), \xi_1(\cdot), ..,\xi_{N_+}(\cdot) $ 
  such that
\begin{equation}
\sup_{t\in\R+} \|u(t,\cdot)-\sum_{j=-N_-\atop j\neq 0}^{N_+} \varphi_{c_j}(\cdot-\xi_j(t)) \|_{\H} <
B(\eps+ L^{-1/8})\;  \label{ini2}
\end{equation}
and 
\begin{equation}
\sup_{t\in\R+} \|u(t,\cdot)-\sum_{j=-N_-\atop j\neq 0}^{N_+} \varphi_{c_j}(\cdot-\xi_j(t)) \|_{L^\infty} \lesssim
\eps^{2/3}+ L^{-\frac{1}{12}}\; .  \label{ini3}
\end{equation}
Moreover, for any $ t\ge 0  $ and $ i\in [[1, N_+]] $(resp. $i\in [[-N_{-},-1]])$, $ \xi_i( t) $ is the only point of maximum (resp. minimum) 
 of $ v(t) $ on $[\xi_i(t)-L/4,\xi_i(t)+L/4] $.
 \end{theorem}

\section{Global well-posedness results} 
We first recall some obvious estimates that will be useful in the sequel of this paper.  
Noticing that $ p(x)=\frac{1}{2}e^{-|x|} $ satisfies $p\ast y=(1-\partial_x^2)^{-1} y $ for any $y\in H^{-1}(\R) $ we easily get
$$
\|u\|_{W^{1,1}}=\|p\ast (u-u_{xx}) \|_{W^{1,1}}\lesssim \| u-u_{xx}\|_{\M}
$$
and
 t$$
\|u_{xx}\|_{\M}\le \|u\|_{L^1}+ \|u-u_{xx} \|_{\M}
$$
 which ensures  that 
\begin{equation} \label{bv}
Y\hookrightarrow  \left\{ u\in W^{1,1}(\R) \mbox{ with } u_x\in {\mathcal BV}(\R) \right\} \; .
\end{equation}
 It is also worth noticing that  for $ u\in C^\infty_0(\R) $, satisfying Hypothesis \ref{hyp},
\begin{equation}\label{formulev}
u(x)=\frac{1}{2} \int_{-\infty}^x e^{x'-x} (u-u_{xx})(x') dx' +\frac{1}{2} \int_x^{+\infty} e^{x-x'} (u-u_{xx})(x') dx'
\end{equation}
and 
$$
u_x(x)=-\frac{1}{2}\int_{-\infty}^x e^{x'-x} (u-u_{xx})(x') dx' + \frac{1}{2}\int_x^{+\infty} e^{x-x'} (u-u_{xx})(x') dx' \; ,
$$
so that for $ x\le x_0 $ we get 
$$
u_x(x) =u(x)-e^{-x} \int_{-\infty}^x e^{x'} y(x') \, dx' \ge u(x)
$$
whereas for $ x\ge x_0 $ we get 
$$
u_x(x) =-u(x)+e^{x} \int^{+\infty}_x e^{-x'} y(x') \, dx' \ge -u(x)
$$
Throughout this paper, we will denote $ \{\rho_n\}_{n\ge 1} $ the mollifiers defined by 
\begin{equation} \label{rho}
\rho_n=\Bigl(\int_{\R} \rho(\xi) \, d\xi 
\Bigr)^{-1} n \rho(n\cdot ) \mbox{ with } \rho(x)=\left\{ 
 \begin{array}{lcl} e^{1/(x^2-1)} & \mbox{for} & |x|<1 \\
0 & \mbox{for} & |x|\ge 1 .
\end{array}
\right.
\end{equation}
Following \cite{E} we approximate $ u\in Y $ satisfying  Hypothesis \ref{hyp} by the  sequence of functions
\begin{equation}\label{app}
u_n=  p\ast y_n \;\text{with}\;y_n=-(\rho_n\ast y^-)(\cdot+\frac{1}{n})+(\rho_n\ast y^+)(\cdot-\frac{1}{n})
 \text{ and } y=u-u_{xx},
\end{equation}
that belong to $Y\cap H^\infty(\R) $  and satisfy Hypothesis \ref{hyp} with the same $ x_0 $.
It is not too hard to check that 
\begin{equation}\label{estyn}
\|y_n\|_{L^1} \le \|y\|_{\mathcal M} .
\end{equation}
Moreover,  noticing that 
$$
u_n  = -\Bigl( \rho_n \ast (p \ast y^-)\Bigr) (\cdot+\frac{1}{n})+\Bigl(\rho_n\ast (p\ast  y^+)\Bigr)(\cdot-\frac{1}{n} ),
$$
with  $ p\ast y^{\mp} \in H^1(\R) \cap W^{1,1}(\R) $,  we infer that 
\begin{equation}\label{mm}
u_n \to u \in H^1(\R) \cap W^{1,1}(\R) \; .
\end{equation}
that ensures that for any $ u\in Y $ satisfying  Hypothesis \ref{hyp} it holds
\begin{equation}\label{dodo}
 u_x \ge u \text{ on } ]-\infty, x_0[\quad  \text{and } \quad  u_x\ge -u \text{ on } ]x_0,+\infty[   \; .
\end{equation}
The following propositions  briefly recall the global well-posedness results for the Cauchy problem of the DP equation 
(see for instance \cite{MR2271927} and \cite{LY} for details of the proof) and its consequences.
 \begin{proposition}{(Strong solutions \cite{LY}, \cite{MR2271927})}\label{smoothWP} \\
 Let $ u_0\in H^s(\R)$ with $ s\ge 3  $. Then the initial value problem \eqref{DP} has a  unique solution $ u\in C\big([0,T]; H^s(\R)\big)  \cap C^1\big([0,T]; H^{s-1}(\R)\big) $ where $ T=T\big(\|u_0\|_{H^{\frac{3}{2}+}}\big)>0 $ and, for any $ r>0 $,  the map $ u_0 \to u $ is continuous  from $ B(0,r)_{H^s} $ into $ C\big([0,T(r); H^s(\R)\big) $.\\
 Moreover, let $ T^* >0 $ be the maximal time of existence of $ u $ in $ H^s(\R)$ then 
 \begin{equation}
 T^*<+\infty \quad \Leftrightarrow  \quad \liminf_{t\nearrow T^*} u_x=-\infty \; .
 \end{equation}
If furthermore $y_0= u_{0}-u_{0,xx}\in L^1(\R) $  and $ u_0  $ satisfies Hypothesis \ref{hyp}  then $ T^*=+\infty $  and 
 $ y=u-u_{xx}\in L^\infty_{loc} \big(\R_+;L^1(\R)\big) $ with 
 
  \begin{equation}\label{estimatey}
\|y(t)\|_{L^1}  \le e^{3t^2 \|u_0\|_{L^2} +2 t \|u_0\|_{L^\infty} } \|y_0\|_{L^1} , \quad\forall t\in \R_+ ,
  \end{equation}
  and
 \begin{equation}\label{conserM}
\int_{\R} y(t,x) \, dx=\int_{\R} y(0,x)\, dx , \quad\forall t\in \R_+ \; .
  \end{equation}
\end{proposition}

\begin{proposition}(Global Weak Solution \cite{MR2271927}) \label{weakGWP}\\
 Let $ u_0 \in Y $ satisfying Hypothesis \ref{hyp}  for some 
  $ x_0\in\R $. \vspace*{2mm}\\
 {\bf 1. Uniqueness and global existence :} \eqref{DP} has a unique solution 
 $$
  u\in  C\big(\R_+; H^1(\R)\big) \cap C^1\big(\R_+;L^2(\R)\big)\cap L^\infty_{loc}\big(\R_+; Y\big) \; .
  $$
 $ M(u)=\langle y, 1 \rangle $, $ E(u)=\langle y, v\rangle   $ and $ F(u) $ are conservation laws . Moreover, for any $ t\in \R_+$, the density momentum 
   $ y(t) $ satisfies   $\supp y^{-}(t)\subset ]-\infty,x_0(t)] $ and $\supp y^+(t) \subset [x_0(t),+\infty[ $ where $ x_0(t)=q(t,x_0) $ with $ q(\cdot,\cdot)$ defined by 
 \begin{equation}\label{defq}
  \left\{ 
  \begin{array}{rcll}
  q_t (t,x) & = & u(t,q(t,x))\, &, \; (t,x)\in \R^2\\
  q(0,x) & =& x\, & , \; x\in\R \; 
  \end{array}
  \right. \; .
  \end{equation} 
  {\bf 2. Continuity with respect to  initial data }: For any sequence $ \{u_{0,n}\} $ bounded in $ Y $ that satisfy Hypothesis \ref{hyp} and such that 
    $ u_{0,n} \to u_0 $ in $ H^1(\R ) $,  the emanating sequence of solutions $ \{u_n\}  $ satisfies for any $ T>0 $
  \begin{equation}\label{cont2}
  u_n \to u \mbox{ in } C\big([0,T]; H^1(\R)\big) \; . \end{equation}
 and
    \begin{equation}\label{cont3}
(1-\partial_x^2) u_n \weaktendsto{n\to\infty} \hspace*{-3mm} \ast \; (1-\partial_x^2) u  \mbox{ in } L^\infty\big(]0,T[; {\mathcal M}(\R)\big)\quad .
   \end{equation}
  
\end{proposition}
\begin{proof} Assertion 1. is proved in  \cite{MR2271927} except the conservation of $F(u)$. But this is clearly a direct consequence of the conservation of $F $ for smooth solutions together with \eqref{cont2}. So let us prove Assertion  2. We first assume that 
$\{u_{0,n}\}  $ is the sequence defined in \eqref{app}.  
In view of  the conservation of $ \H $ and \eqref{estimatey}, the sequence $ \{u_n\} $ of smooth solutions to the DP equation emanating from $\{u_{0,n}\} $ is bounded in $ C([0,T]; H^1) \cap L^\infty(]0,T[; Y) $ for any fixed $ T>0 $.  
 Therefore, there exists $ w\in  L^\infty(\R_+;H^1(\R)) $ with $ (1-\partial_x^2) w \in 
   L^\infty_{loc}(\R_+; {\mathcal M(\R)}) $ such that, for any $ T>0$,
   $$
     u_n \weaktendsto{n\to\infty} w \in L^\infty(]0,T[; H^1(\R))  \quad \mbox{and}\quad (1-\partial_x^2) u_n \weaktendsto{n\to\infty} \hspace*{-3mm} \ast \; (1-\partial_x^2) w  \mbox{ in } L^\infty(]0,T[; {\mathcal M}(\R))  \; .
   $$
  Moreover, in view of \eqref{DP}, $ \{\partial_t u_n\} $ is bounded in $L^\infty(]0,T[; L^2(\R) \cap L^1(\R) )$ and Helly's,  Aubin-Lions compactness and Arzela-Ascoli theorems  ensure that $ w $ is a solution to \eqref{DP} that belongs to $ C_{w}([0,T]; H^1(\R)) $ with  $ w(0)=u_0 $. In particular, $ w_t\in L^\infty(]0,T[; L^2(\R)) $ and thus $ w\in C([0,T];L^2(\R)) $. Since $ w\in L^\infty(]0,T[; H^{\frac{3}{2}-} (\R))$, this  actually implies that 
   $ w\in C([0,T];  H^{\frac{3}{2}-}(\R)) $ and thus $ w_t\in C([0,T] ; L^2(\R))$. Therefore, 
 $w$  belongs to the uniqueness class which ensures that $ w=u$ and  that \eqref{cont3} holds for this sequence. In particular passing to the limit in \eqref{estimatey} we infer that for any $ u_0\in Y $ satisfying Hypothesis \ref{hyp} it holds
 \begin{equation}\label{estimatey2}
\|y(t)\|_{\M}  \le e^{3t^2 \|u_0\|_{L^2} +2 t \|u_0\|_{L^\infty} } \|y_0\|_{\M} , \forall t\in \R_+ \; .
  \end{equation}
With \eqref{estimatey2} in hands, we can now proceed exactly in the same way but for any sequence $ \{u_{0,n}\}$
 bounded in  $ Y $  that converges to $ u_0 $ in $H^1(\R) $. This shows that  \eqref{cont3} holds. 
  Finally, the conservation of $ E(\cdot) $ together with  the  weak convergence result in  $C_{w}([0,T]; H^1(\R)) $ 
   lead to  a strong convergence result in $C([0,T]; L^2(\R)) $ that leads to \eqref{cont2} by using that 
   $u\in L^\infty_{loc} (\R_+; H^{\frac{3}{2}-} (\R))$.
\end{proof}

In the sequel, we will make a constant use of the following properties of the flow-map $q(\cdot,\cdot) $ established for instance in 
\cite{ZYin} : Under the hypotheses of Proposition \ref{weakGWP}, 
\begin{enumerate}
\item 
The mapping $q(t,\cdot)$ is an increasing diffeomorphism of $\R$ with
\begin{equation}\label{13}
q_x(t,x)=exp\Big(\int_0^tu_x\big(s,q(s,x)\big)ds\Big)>0,\qquad\forall(t,x)\in\R_{+}\times\R.
\end{equation}
 \item  If moreover $u_0\in H^3(\R) $  then  
\begin{equation}\label{13b}
y\big(t,q(t,x)\big)q_x^3(t,x)=y_0(x),\qquad\forall(t,x)\in\R_{+}\times\R.
\end{equation}
In particular, for all $ t\ge 0 $, 
\begin{equation}\label{13c}
y\big(t,x_0(t))\big)=y\big(t,q(t,x_0)\big)=0 ,\qquad\forall(t,x)\in\R_{+}\times\R.
\end{equation}
\end{enumerate}



\section{Some uniform $ L^\infty $-estimates}
In \cite{LY} it is proven that as far as the solution to the DP equation stays smooth, its $ L^\infty $-norm can be bounded by a polynomial function of time with coefficients that depend only on the $ L^2 $ and $ L^\infty $-norm of the initial data. In this section we first improve this result under  Hypothesis 1 by showing that the solution is then  bounded in positive times by a constant that only depends on the $L^2$-norm of the initial data. This result is not directly needed in our work but we think that it has its own interest.  In a second time we use the same type of arguments to prove that any function that is $ L^2$-close to a peakon profile and satisfies Hypothesis 1, is actually $ L^\infty $-close to the peakon profile. This last result  will be very useful for our work and will for instance enable us to prove that as far as $u $ stays $ L^2$-close to a translation of a peakon profile, the growth of the total variation of its momentum density can be control by an exponential function of the  time but with a small constant in front of the time. This will be the aim of the last lemma of this section. 

\begin{lemma}\label{LemmaPP0}
For any  $u_0\in Y$ satisfying Hypothesis 1,   the associated solution $ u\in C(\R_+; H^1) $ to \eqref{DP}  given by Proposition \ref{weakGWP} satisfies
\begin{equation}\label{u0uinf}
\Vert u(t)\Vert_{L^{\infty}(\R)}\leq 2(1+\sqrt{2})\Vert u_0\Vert_{\mathcal{H}}, \quad \forall t\in \R_+.
\end{equation}
\end{lemma}
\begin{proof}
We fix $t\in \R_+ $, $ x\in \R $  
 and denote by $\mathcal{E}(x)$ the integer part of $x$. Since $u(t,\cdot)\in H^1(\R)\hookrightarrow C(\R)$, 
  the Mean-Value theorem for integrals together with \eqref{1.5} and the conservation of $ \|\cdot\|_{\H} $ ensure that there exists  $\eta\in \big[\mathcal{E}(x)-1,\mathcal{E}(x)\big]$ such that
$$
u^2(t,\eta)=\int_{\mathcal{E}(x)-1}^{\mathcal{E}(x)}u^2(t,\theta)d\theta\leq\Vert u(t,\cdot)\Vert_{L^2(\R)}^2\leq4\Vert u(t,\cdot)\Vert_{\mathcal{H}}^2=4\Vert u_0\Vert_{\mathcal{H}}^2.
$$
Therefore,  by \eqref{dodo} and \eqref{1.5}, since $ 0\le x-\eta\le 2 $, one may write
\begin{eqnarray}
u(t,x)& = & u(t,\eta)+\int^{x}_{\eta}u_{\theta}(t,\theta)d\theta\geq  -2\Vert u_0\Vert_{\mathcal{H}}
-\int_{\eta}^x |u(t,\theta)| \, d\theta \geq  -2\Vert u_0\Vert_{\mathcal{H}}
-2\sqrt{x-\eta}\Vert u_0\Vert_{\mathcal{H}}\nonumber \\
& \geq&  -2(1+\sqrt{2})\Vert u_0\Vert_{\mathcal{H}} \; .\label{e3}
\end{eqnarray}
Now, suppose that there exists $x_{\ast}\in\R$ such that $u(t,x_{\ast})>2(1+\sqrt{2})\Vert u_0\Vert_{\mathcal{H}}$.  Then, on one side the Mean-Value theorem for integrals similarly ensures that there exists  $\gamma\in \big[\mathcal{E}(x_{\ast})+1,\mathcal{E}(x_{\ast})+2\big]$ such that
$$
u^2(t,\gamma)=\int_{\mathcal{E}(x_{\ast})+1}^{\mathcal{E}(x_{\ast})+2}u^2(t,\theta)d\theta\leq\Vert u(t,\cdot)\Vert_{L^2(\R)}^2\leq4\Vert u_0\Vert_{\mathcal{H}}^2\; .
$$
On the other side, \eqref{dodo} again leads to 
\begin{equation}
u(t,\gamma)=u(t,x_{\ast})+\int_{x_{\ast}}^{\gamma}u_{\theta}(t,\theta)d\theta>2(1+\sqrt{2})\Vert u_0\Vert_{\mathcal{H}}-2\sqrt{\gamma-x_{\ast}}\Vert u_0\Vert_{\mathcal{H}}>2\Vert u_0\Vert_{\mathcal{H}}\; .
\end{equation}
The fact that  the two above estimates are not compatible completes the proof of the lemma.
\end{proof}

\begin{lemma}[$L^{\infty}$  approximations]\label{LemmaPP1}
Let  $ \psi\in W^{1,\infty}(\R)\cap L^2(\R) $  and  $u\in Y $, satisfying Hypothesis 1, then  
\begin{equation}
\|u-\psi\|_{L^{\infty}(\mathbb{R})}\le 2 \|u-\psi\|_{\mathcal{H}}^{2/3} \big(1+  \sqrt{2} \|u-\psi\|_{\mathcal{H}}^{2/3}+ 
\|\psi\|_{L^\infty}+\|\psi'\|_{L^\infty}\big) \; .
\label{PP22}
\end{equation}

In particular, for any $(c,r) \in \R^2 $ it holds 
\begin{equation}
\|u-\varphi_{c}(\cdot-r)\|_{L^{\infty}(\mathbb{R})}\le 2 \|u-\varphi_{c}(\cdot-r)\|_{\mathcal{H}}^{2/3} \big(1+  \sqrt{2} \|u-\varphi_{c}(\cdot-r)\|_{\mathcal{H}}^{2/3}+ 2 c\big) \; .
\label{PP2}
\end{equation}
\end{lemma}
\noindent
\begin{proof}
We first notice that \eqref{PP2} follows directly from \eqref{PP22} by taking $ \psi=\varphi_c(\cdot-r) $ and using that 
$\|\varphi_c\|_{L^\infty}=\|\varphi_c'\|_{L^\infty}=c$.

Let us now prove \eqref{PP2}. 
We set $ \alpha=\|u-\psi\|_{\mathcal{H}}^{2/3} $.
Fixing  $x\in\R$, there exists $ k\in \Z $ such $x\in  \big[k\alpha,(k+1)\alpha\big[$. Therefore, applying the Mean-Value theorem on the interval   $\big[(k-1)\alpha,k\alpha\big]$, we obtain that there exists  $\eta\in\big[(k-1)\alpha,k\alpha\big]$ such that
\begin{equation}\label{contractss}
\big[u(\eta)-\psi(\eta)\big]^2=\frac{1}{\alpha}\int_{(k-1)\alpha}^{k\alpha}\big[u(\theta)-\psi(\theta)\big]^2d\theta\leq\frac{4}{\alpha}\Vert u-\psi\Vert_{\mathcal{H}}^2=4 \alpha^2.
\end{equation}
Now,  in view of  \eqref{dodo}, we get
\begin{equation}
u(x)-\psi(x)=u(\eta)-\psi(\eta)+\int^{x}_{\eta}[u_{x}(\theta)-\psi'(\theta)]d\theta\geq-2\alpha-\sqrt{2\alpha}\, 
\big\Vert\vert u\vert+|\psi'|\big\Vert_{L^2(](k-1)\alpha,(k+1)\alpha[)}.
\end{equation}
and  the triangular inequality together with \eqref{1.5}  yield
$$
\big\Vert |u|+|\psi' \vert\big\Vert_{L^2(](k-1)\alpha,(k+1)\alpha[)} \le 
  \big\Vert |u-\psi |+|\psi|+|\psi'|\big\Vert_{L^2(](k-1)\alpha,(k+1)\alpha[)}\le 2 \|u-\psi\|_{\H} +
  \sqrt{2\alpha} \, (\|\psi\|_{L^\infty}+\|\psi'\|_{L^\infty}).
$$
We thus eventually   get 
\begin{equation}
u(x)-\psi(x)\geq-2\alpha \big(1+  \sqrt{2}\alpha  + \|\psi\|_{L^\infty}+\|\psi'\|_{L^\infty}\big).
\end{equation}
Now, suppose  that there exists $x_{\ast}\in\R$ such that  
$$
u(x_{\ast})-\psi(x_{\ast})>2 \alpha \big(1+  \sqrt{2} \alpha+\|\psi\|_{L^\infty}+\|\psi'\|_{L^\infty}\big)
$$
 Similarly,  there exists $k_{\ast}\in\R$ such that $x_{\ast}\in\big[k_{\ast}\alpha,(k_{\ast}+1)\alpha\big[$ and applying the Mean-Value theorem for integrals on $\big[(k_{\ast}+1)\alpha,(k_{\ast}+2)\alpha\big]$  we obtain that  there exists   $\gamma\in\big[(k_{\ast}+1)\alpha,(k_{\ast}+2)\alpha\big]$ such that, on one hand, 
$$
\big[u(\gamma)-\psi(\gamma)\big]^2=\frac{1}{\alpha}\int_{(k_{\ast}+1)\alpha}^{(k_{\ast}+2)\alpha}\big[u(\theta)-\psi(\theta)\big]^2d\theta 
 \le  4 \alpha^2 \; . $$
On the other hand, proceeding as above  we get
\begin{align*}
u(\gamma)-\psi(\gamma) & =u(x_{\ast})-\psi(x_{\ast})+\int_{x_{\ast}}^{\gamma}[u_{x}(\theta)-\psi'(\theta)]d\theta\\
& >2\alpha \Big(1+  \sqrt{2} \alpha+ \|\psi\|_{L^\infty}+\|\psi'\|_{L^\infty}\Big)-\sqrt{2\alpha}\Bigl(2\alpha^{3/2}+ 
\sqrt{2\alpha} (\|\psi\|_{L^\infty}+\|\psi'\|_{L^\infty}) \Bigr)>  2 \alpha \; .
\end{align*}
The incompatibility of the two above estimates completes the proof of  the lemma.
\end{proof}
\begin{lemma}
Let $u_0\in Y $ satisfying Hypothesis \ref{hyp} and $ u\in C(\R_+;H^1) \cap L^\infty(\R_+; Y) $ be the associated solution to DP given by Proposition \ref{weakGWP}.
If for some $ c\ge 0$, $0<\alpha<1 $ and $ T>0 $ it holds 
\begin{equation}\label{assum1}
\sup_{t\in[0,T]}\inf_{r\in\R} \Vert u(t,\cdot)-\varphi_c(\cdot-r)\Vert_{\mathcal{H}}\leq\alpha,
\end{equation}
then 
\begin{equation}\label{50b}
u(t)\ge -4\alpha^\frac{2}{3}(2+c),  \quad \forall t\in  [0,T] 
\end{equation}
and
\begin{equation}\label{50}
\sup_{t\in [0,T]} \Vert y(t)\Vert_{\M}\leq \big(1+2e^{8\alpha^\frac{2}{3}(2+c)t}\big)\Vert y_0\Vert_{\M}.
\end{equation}
\end{lemma}
\begin{proof}
According to Proposition \ref{weakGWP},  approximating $ u_0 $ by the sequence $u_{0,n} $ given by \eqref{app},  it suffices to prove the result 
 for smooth initial data $ u_0\in  Y\cap H^\infty(\R)$ satisfying Hypothesis \ref{hyp}.
We notice that since $ \varphi_c>0 $ on $ \R $, \eqref{assum1} together with Lemma \ref{LemmaPP1} ensure that 
 for all $ t\in [0,T]$, 
$$
u(t,\cdot)\ge -2\alpha^{2/3} \big(1+  \sqrt{2} \alpha^{2/3}+ 2 c\big)\ge -4\alpha^\frac{2}{3}(2+c)  \quad \text{on}\;\R \; .
$$
Therefore, according to  \eqref{DP2}, \eqref{defq}, \eqref{13c} and \eqref{dodo}, we have
\begin{align*}
\frac{d}{dt}\int_{\R}y^{-}(t,x)dx=-\frac{d}{dt}\int_{-\infty}^{q(t,x_0)}y(t,x)dx&=2\int_{-\infty}^{q(t,x_0)}u_x(t,x)y(t,x)dx\\
&\leq-2\int_{-\infty}^{q(t,x_0)}\left(-u(t,x)\right)y(t,x)dx\\
&\leq 8  \alpha^\frac{2}{3}(2+c)  \int_{\R}y^{-}(t,x)dx.
\end{align*}
Hence, Gr$\ddot{\text{o}}$nwall's inequality yields
\begin{equation}
\int_{\R}y^{-}(t,x)dx\leq e^{8\alpha^\frac{2}{3} (2+c)  t}\int_{\R}y^{-}_0(x)dx.
\end{equation}
Moreover, since, according to Proposition \ref{smoothWP},  $\dsp M(u)=\int_{\R} y $ is conserved for positive times, it holds $\dsp \int_{\R}y^{+}(t,x)dx=\int_{\R}y_0(x)dx+\int_{\R}y^{-}(t,x)dx$ and thus  
\begin{equation}\label{l1y}
\Vert y(t,\cdot)\Vert _{L^1(\R)}\leq \big(1+2e^{8\alpha^\frac{2}{3}(2+c)  t}\big)\Vert y_0\Vert_{L^1(\R)}.
\end{equation}
\end{proof}


\section{A dynamic estimate on $y^-$}\label{section4y-}
In this section we assume that $\sup_{t\in [0,T]} \inf_{r\in\R} \|u(t)-\varphi_c(\cdot-r)\|_{\H} <\varepsilon $
for some $ T>0 $ and some $ 0<\varepsilon<1  $ small enough.  Then we can construct a $ C^1$-function $ x \, :\, [0,T] \to \R $ such that
$ \sup_{t\in [0,T]}  \|u(t)-\varphi_c(\cdot-x(t))\|_{\H} \lesssim \varepsilon $ and we study the behavior of $ y^- $ in an  growing with time interval at the left of $x(t) $. 
\begin{lemma}\label{modulation}
There exist $ \tilde{\eps}_0>0 $, $0<\kappa_0<1 $ and $ K\ge 1 $ such that if a solution $ u\in C([0,T]; H^1(\R)) $ to 
\eqref{DP} satisfies for some $ c>0 $ and some function $ r \, : [0,T] \to \R $, 
\begin{equation}\label{assum3}
\sup_{t\in[0,T]} \Vert u(t,\cdot)-\varphi_c(\cdot-r(t))\Vert_{\mathcal{H}}< c \tilde{\eps}_0,
\end{equation}
then  there exist a  unique function $x\colon[0,T]\longrightarrow\R$ such that
\begin{equation}\label{nanaa}
\sup_{t\in[0,T]}\big\vert x(t)-r(t)\big\vert\le \kappa_0<\ln(3/2)
\end{equation}
and 
\begin{equation}\label{orto}
\int_{\R} v(t,x) \rho'(x-x(t)) \, dx =0 , \quad \forall t\in [0,T] \; .
\end{equation}
where $ v=(4-\partial_x^2)^{-1} u $ and $\rho=(4-\partial_x^2)^{-1}\varphi $. 
Moreover, $ x(\cdot) \in C^1([0,T]) $ with 
\begin{equation}\label{nanaa2}
\sup_{t\in[0,T]}\big\vert \dot{x}(t)-c\big\vert\leq \frac{c}{8},
\end{equation}
and if 
\begin{equation}\label{nana}
\sup_{t\in[0,T]}\big\Vert u(t,\cdot)-\varphi_c(\cdot-r(t))\big\Vert_{\H}< \eps ,
\end{equation}
 for some $ 0<\varepsilon\le  c \tilde{\eps}_0 $ then 
\begin{equation}\label{64}
\sup_{t\in[0,T]}\big\Vert u(t,\cdot)-\varphi_c(\cdot-x(t))\big\Vert_{\mathcal{H}}\leq K \varepsilon \; .
\end{equation}
\end{lemma}

\begin{proof}
We  follow  the same approach as in \cite{AK}, by requiring an orthogonality condition on $ v=(4-\partial_x^2)^{-1} u $ instead of $ u$.  This will be useful to get the $ C^1$-regularity of $ x(\cdot) $. In the sequel of the proof, we endow $ H^2(\R) $ with the norm (that is equivalent to the usual norm)
$$
\|v\|_{H^2}^2:=\int_{\R} 4 v^2+5v_x^2+v_{xx}^2 = \|(4-\partial_x^2) v \|_{\H}^2 
$$
where the last identity follows from \eqref{i}.
Let $0<\eps<1$. For $r\in\R$ we introduce the function $Y\colon (-\eps,\eps)\times B_{H^2}\big(\rho(\cdot-r),\eps\big)\longrightarrow\R$ defined by
$$
Y(y,v)=\int_{\R}\big[v(t,x)-\rho(x-r-y)\big]\rho'(x-r-y)dx.
$$
It is clearly that $Y\big(0,\rho(\cdot-r)\big)=0$ and that $Y$ is of class $C^1$. Moreover, by integration by parts,  it holds
$$
\frac{\partial Y}{\partial y}(y,v)=-\int_{\R}v(t,x)\rho''(x-r-y)dx.
$$
Hence, by integration by parts we may write
\begin{equation}\label{eded}
\frac{\partial Y}{\partial y}\big(0,\rho(\cdot-r)\big)=-\int_{\R}\rho(x-r)\partial_x^2\rho(x-r)dx=\Vert \partial_x\rho(\cdot-r)\Vert_{L^2(\R)}^2=\frac{5}{54}\neq0.
\end{equation}
From the Implicit Function Theorem  we deduce that there exist $\tilde{\eps}_0>0$, $0<\kappa_0<\ln(3/2)$ and a $C^1$-function $y_r\colon B_{H^2}\big(\rho(\cdot-r),\tilde{\eps}_0\big)\longrightarrow]-\kappa_0,\kappa_0[$ which is uniquely determined such that
$$
Y\big(y_r(v),v\big)=Y\big(0,\rho(\cdot-r)\big)=0,\qquad \forall v\in B_{H^2}\big(\rho(\cdot-r),\tilde{\eps}_0\big).
$$
In particular, there exists $C_0>0$ such that if $v\in B_{H^2}\big(\rho(\cdot-r),\beta\big)$, with $0<\beta\leq\tilde{\eps}_0$, then 
\begin{equation}\label{ycb}
\big\vert y_r(v(t,\cdot))\big\vert\leq C_0\beta.
\end{equation}
Note that by a translation symmetry argument $\tilde{\eps}_0$, $\kappa_0$, and $C_0$ are independent of $r\in\R$. Therefore, by uniqueness, we can define a $C^1$-mapping $\tilde{x}\colon \bigcup_{r\in\R}B_{H^2}\big(\rho(\cdot-r),\tilde{\eps}_0\big)\longrightarrow]-\kappa_0,\kappa_0[$ by setting
$$
\tilde{x}(v)=r+y_r(v)\qquad\forall v\in B_{H^2}\big(\rho(\cdot-r),\tilde{\eps}_0\big).
$$
Now, according to \eqref{assum3}, it holds $ \{\frac{1}{c}v(t,), [0,T]\} \subset \cup_{z\in\R} B_{H^2}(\rho (\cdot-z),\tilde{\varepsilon}_0)$  so that we can define  the function $ x(\cdot)$ on $\R$ by setting $ x(t)=\tilde{x}(v(t)) $. By construction $ x(\cdot) $ satisfies  \eqref{nanaa}-\eqref{orto}.
 Moreover, \eqref{nana} together with \eqref{ycb} ensure that for any $ c>0 $ and any $
 0<\varepsilon<c \tilde{\eps}_0 $, it holds 
 \begin{equation}\label{vb1}
 \|\frac{1}{c}u(t)-\varphi(\cdot-x(t))\|_{\H}  \le   (\frac{\varepsilon}{c}) + \sup_{|z|\le C_0 \frac{\varepsilon}{c}} \|\varphi-\varphi(\cdot-z))\|_{\H}  
  \lesssim   \frac{\varepsilon}{c}
 \end{equation}
which proves \eqref{64}.

  In view of \eqref{DP}, any solution $ u\in C(\R;H^1(\R)) $ of (D-P) satisfies $ u_t\in C(\R;L^2(\R)) $ and thus belongs to $C^1(\R;L^2(\R)) $. This ensures that $ v\in C^1(\R;H^2(\R))$ so that   the mapping $ t\mapsto x(t)=\tilde{x}(v(t)) $ is of class $ C^1$ on $ \R $.  
  
  Now, we notice that applying the operator $(4-\partial_x^2)^{-1} $ to the two members of \eqref{DP} and using that 
\begin{equation}\label{zq}
(4-\partial_x^2)^{-1} (1-\partial_x^2)^{-1}= \frac{1}{3} (1-\partial_x^2)^{-1}-  \frac{1}{3} (4-\partial_x^2)^{-1}\; , 
\end{equation}
we get that $ v$ satisfies
\begin{equation}\label{vt}
v_t=-\frac{1}{2} \partial_x (1-\partial_x^2)^{-1} u^2\; .
\end{equation} 
 On the other hand, 
  setting  $ R(t,\cdot)=c \rho (\cdot-x(t))$ and $ w=v-R$ and differentiating \eqref{orto} with respect to time we get 
  \begin{eqnarray}\label{vb}
   \int_{\R} w_t \rho'(\cdot-x(t))&=& \dot{x}(t) \int_{\R} w\,  \rho''(\cdot-x(t)) \nonumber \\
   & = & -\dot{x}(t) \int_{\R} \partial_x w \, \rho'(\cdot-x(t))\nonumber \\
   &=& (\dot{x}(t)-c)O(\|w\|_{H^1})+ c \, O(\|w\|_{H^1})\; .
  \end{eqnarray}
  Substituting $ v $ by $ w+R $ in \eqref{vt} and using that $ R $ satisfies
  $$
  \partial_t R +(\dot{x}-c) \partial_x R =-\frac{1}{2}\partial_x (1-\partial_x^2)^{-1}  \varphi_c^2(\cdot-x(t))  \; ,
  $$
  we infer that $ w$ satisfies 
  $$
  w_t -(\dot{x}-c) \partial_x R = -\frac{1}{2} \partial_x (1-\partial_x^2)^{-1} \Bigl( u^2-\varphi_c^2\Bigr)
  =-\frac{1}{2} \partial_x (1-\partial_x^2)^{-1} \Bigl( (u-\varphi_c)(u+\varphi_c)\Bigr) \; .
  $$
  Taking the $ L^2$-scalar product of this last equality with $\rho'(\cdot-x(t)) $ and using \eqref{vb} together with \eqref{assum3} and \eqref{64} we get, for all $ t\in [0,T] $, 
  $$
  \Bigl|(\dot{x}(t)-c) \Bigl(\int_{\R}\,  \partial_x R(t,\cdot) \,  \rho'(\cdot-x(t))+c\, O(\|w\|_{H^1})\Bigr) \Bigr|\le  O(\|w\|_{H^1})+O(\Vert u-\varphi_c(\cdot-x(t)\Vert_{\mathcal{H}}) \lesssim K c\, \tilde{\eps}_0\; .
  $$
  Therefore, by recalling \eqref{eded} and   possibly decreasing the value of  $ \tilde{\eps}_0>0 $ so that $ K\tilde{\eps}_0 \ll 1$,
   we obtain \eqref{nanaa2}. 
\end{proof}
\begin{proposition}\label{pro2}
There exists $\eps_0>0 $ such that for any  $u_0\in Y\cap H^\infty(\R) $ satisfying Hypothesis \ref{hyp}, if  the solution $u\in C(\R_{+}; H^\infty(\mathbb{R})\big)$ emanating from $u_0$ satisfies for some
 $ c>0  $, $T >0 $ and some function $ r \, : [0,T] \to \R $,
\begin{equation}\label{assum4}
\sup_{t\in[0,T]} \Vert u(t,\cdot)-\varphi_c(\cdot-r(t))\Vert_{\mathcal{H}}<   \eps_0(1\wedge c^2), 
\end{equation}
  then
 for all $t\in[0,T]$,
\begin{equation}\label{81}
\Vert y^{-}(t,\cdot)\Vert_{L^1(]r(t)-\frac{1}{16}ct,+\infty[)}\leq e^{-ct/8}\Vert y_0\Vert_{L^1(\R)}.
\end{equation}
where $y^- =\max(-y,0)$ and $x(\cdot)$ is the $C^1$-function 
constructed in Lemma \ref{modulation}.
\end{proposition}
\begin{proof}
  Let $ \tilde{\eps}_0 >0 $ and $ K\ge 1 $  be the universal constants that appears in the statement of Lemma \ref{modulation}. Assuming  \eqref{nana} with 
$$
 \eps <\min(c \tilde{\eps}_0, 10^{-20} (1\wedge c^{2})/K)\le \frac{10^{-20}\wedge \tilde{\eps}_0}{K}  (1\wedge c^{2})\; , 
 $$
  \eqref{64}, Lemma \ref{LemmaPP1}  ensure that
\begin{equation}\label{tf1}
\sup_{t\in[0,T]} \Vert u(t,\cdot)-\varphi_c(\cdot-x(t))\Vert_{L^\infty}\leq  10^{-5} c  \; .
\end{equation}
where $ x(\cdot) $ is the $ C^1$-function constructed in Lemma \ref{modulation}.
Therefore, setting
\begin{equation}\label{choice}
\eps_0:= \frac{10^{-20}\wedge \tilde{\eps}_0}{K}   \; ,
\end{equation}
\eqref{assum4} ensures that \eqref{tf1} holds.

Let $ t\in [0,T]$, we separate two possible cases  according to the distance between $x_0(t/2) $  and $ x(t/2) $, where  $ x_0(\cdot) $ is defined in Proposition \ref{weakGWP}.\\
\noindent
{\it Case 1:}
\begin{equation}\label{79}
x_0(t/2)<x(t/2)-\ln(3/2).
\end{equation}
In view of   \eqref{tf1}  and the monotony of $ \varphi $ on $ \R_-$, it holds 
\begin{equation}
u(\tau,x)\leq\varphi_c\big(-\ln(3/2)\big)+\frac{c}{16}=\frac{2}{3}c+\frac{1}{16}c\leq\frac{3}{4}c,
\quad \forall   x\le x(\tau)-\ln(3/2)  \;\text{with}\; \tau\in [0,T]\; .
\end{equation}
In particular   \eqref{79} and   \eqref{defq} lead to
\begin{equation}
\dot{x}_0(t/2)=u\big(t/2,x_0(t/2)\big)\leq\frac{3}{4}c.
\end{equation}
Therefore, since  \eqref{nanaa2}  forces $\dot{x}(t)\ge 7c/8$ on $[0,T]$, a classical continuity argument 
 ensures that $x_0(\cdot)<x(\cdot)-\ln(3/2) $ on  $[t/2,T] $ and thus $ \dot{x}_0(\cdot) \leq\frac{3}{4}c $ on $ [t/2,T] $. 
It follows from \eqref{nanaa} that
$$
r(t)-x_0(t)\ge  x(t)-x_0(t)-\ln(3/2) =\int_{t/2}^t\big[\dot{x}(\theta)-\dot{x}_0(\theta)\big]d\theta+x(t/2)-x_0(t/2)-\ln(3/2) \\
 \geq \frac{c}{16}t \; .
$$
This proves that $y^{-}(t,\cdot)=0$ on $]r(t)-\frac{1}{16}ct,+\infty[$ and thus that  \eqref{81} holds in this case.\\
\noindent
{\it Case 2:}
\begin{equation}\label{83}
x_0(t/2)\geq x(t/2)-\ln(3/2).
\end{equation}
Then we first claim that 
\begin{equation}
x_0(\tau)\geq x(\tau)-\ln(3/2)\qquad\forall\text{$\tau\in[0,t/2]$.}
\label{83'}
\end{equation}
Indeed, assuming the contrary, we would get  as above that $ x_0(\cdot)<x(\cdot)-\ln(3/2) $ on $[\tau,T] $ that would 
contradicts \eqref{83}. Second, we notice that \eqref{tf1} ensures that 
\begin{equation}\label{83''}
u\big(\tau,x(\tau)-\ln(3)\big) \ge \varphi_c(-\ln(3)) - \frac{c}{16}\ge \frac{c}{4}  , \quad \forall \tau\in [0,T] \;.
\end{equation}
Since \eqref{dodo} forces $ u_x(\tau)\ge u(\tau) $ on $]-\infty, x_0(\tau)]$ for any $ \tau \in \R_+$, \eqref{83'}-\eqref{83''} then ensure that 
$ u(\tau) $ is increasing on $[x_0(\tau)-\ln(2),x_0(\tau)] $ and 
\begin{equation}\label{81v}
u_x(\tau, x) \ge u(\tau,x) \ge\frac{c}{4}  , \quad \forall (\tau,x)\in [0,T]\times [x_0(\tau)-\ln 2, x_0(\tau)] \; .
\end{equation}
Now, in this case we divide the proof into two steps.\\
\noindent
{\it Step$\colon1$}. The aim of this step is to prove the following estimate on $ y(t/2)$ :
\begin{equation}\label{91}
\left\vert \int_{x_0(t/2)-\ln 2}^{x_0(t/2)}y(t/2,s)ds\right\vert\leq e^{-\frac{1}{4}ct}\Vert y_0\Vert_{L^1(\R)}.
\end{equation}
For $ \tau\in \R_+ $, we denote by $ q^{-1}(\tau,\cdot) $ the inverse mapping of $ q(\tau,\cdot) $. Then, the change of variables along the flow $ \theta=q^{-1}(t/2,s) $ leads to 
\begin{equation}\label{1001}
\int_{x_0(t/2)-\ln 2}^{x_0(t/2)}y(t/2,s)ds=\int_{q^{-1}(t/2,x_0(t/2)-\ln 2)}^{q^{-1}(t/2,x_0(t/2))}y\big(t/2,q(t/2,\theta)\big)q_x(t/2,\theta)\, d\theta.
\end{equation}
Since $ x_0(\tau)=q(\tau,x_0) $ it clearly holds  $x_0=q^{-1}(t/2,x_0(t/2)) $ and  \eqref{81v} together with \eqref{defq} force 
$$
\partial_t q\big(\tau,q^{-1}(t/2,x)\big) \le\dot{x}_0(\tau) , \quad \forall \hspace{0.5mm}(\tau,x)\in [0,t/2]\times [x_0(t/2)-\ln 2, x_0(t/2)] \; .
$$
This ensures that for all $\tau\in [0,t/2] $, 
\begin{equation}\label{1002}
 0<  x_0(\tau)-q(\tau,q^{-1}(t/2, x_0(t/2)-\ln 2))\le    x_0(t/2)-q(t/2,q^{-1}(t/2, x_0(t/2)-\ln 2))=\ln 2 
\end{equation}
In particular,  for any $ \theta \in [q^{-1}(t/2, x_0(t/2)-\ln 2), x_0] $ and any $\tau\in [0,t/2] $, it holds 
$q(\tau,\theta)\in [x_0(\tau)-\ln 2,x_0(\tau)] $ and \eqref{81v} yields
$$
u_x(\tau,q(\tau,\theta))\ge u(\tau,q(\tau,\theta))\ge c/4 \; .
$$
In view of \eqref{13} we thus deduce that 
$$
q_x(t/2,\theta)=\exp\Bigl(\int_0^{t/2} u_x(\tau, q(\tau,\theta))\, d\tau\Bigr) \ge \exp(\frac{c}{8} t) \; .
$$
Plugging this estimate in \eqref{1001},  using \eqref{13b}, \eqref{1002} and that $ y(\tau,\cdot) \le 0 $ on 
 $ ]-\infty,x_0(\tau)] $ for  $\tau\ge 0 $, we eventually get 
\begin{eqnarray*}
\int_{x_0(t/2)-\ln 2}^{x_0(t/2)}y(t/2,s)ds & \ge &
e^{-\frac{c}{4} t }  \displaystyle \int_{q^{-1}(t/2,x_0(t/2)-\ln 2)}^{q^{-1}(t/2,x_0(t/2))=x_0}y\big(t/2,q(t/2,\theta)\big)q_x^3(t/2,\theta)\, d\theta  \\
& \ge &e^{-\frac{c}{4} t }  \int_{x_0-\ln 2}^{x_0} y(0,\theta) \, d\theta 
\end{eqnarray*}
which proves \eqref{91}. 

{\it Step$\colon2$}. In this step,  we prove that 
\begin{equation}\label{94}
\left\vert \int_{x(t)-\ln (3/2)-\frac{c}{16} t }^{x_0(t)}y(t,s)\, ds\right\vert \le 
e^{c t/8} \left\vert \int_{x_0(t/2)-\ln 2}^{x_0(t/2)}y(t/2,s)\, ds\right\vert
\end{equation}
Clearly, \eqref{94} combined with \eqref{91}  and \eqref{nanaa} prove that  \eqref{81} also holds  in this case which completes the proof of the proposition. 

First, for any $t_{1}\geq0$ we define the function $q_{t_1}(\cdot,\cdot)$ on $\R_{+}\times\R$ as follows
\begin{equation}\label{defq1}
\left\{
\begin{array}{lcl}
\displaystyle \partial_tq_{t_1}(t,x)=u(t,q_{t_1}(t,x)), &\forall(t,x)\in\R_{+}\times\R,\vspace{1mm}\\
\displaystyle q_{t_1}(t_1,x)=x, & x\in\R.
\end{array}
\right.
\end{equation}
 The mapping $q_{t_1}(t,\cdot)$ is an increasing diffeomorphism of $\R$ and we  denote by
 $q^{-1}_{t_1}(t,\cdot)$ it inverse mapping. As in \eqref{13} we have
$$
\partial_xq_{t_1}(t,x)=exp\Big(\int_{t_1}^tu_x\big(s,q_{t_1}(s,x)\big)ds\Big)>0,\qquad\forall(t,x)\in\R_{+}\times\R,
$$
and
\begin{equation}\label{13aaa}
y\big(t,q_{t_1}(t,x)\big)(\partial_xq_{t_1})^3(t,x)=y(t_1,x),\qquad\forall(t,x)\in\R_+\times\R.
\end{equation}
In particular, \eqref{50b},  \eqref{dodo} together with \eqref{tf1} ensure that for any $ \tau\in [t/2, t] $ and any  $x\le x_0(t/2) $, 
\begin{equation}\label{13bbb}
\partial_x q_{t/2}(\tau,x)\ge exp\Big(-\int_{t/2}^t  2^{-5} c \,  ds\Big)\ge e^{-2^{-4} c t } \quad .
\end{equation}
  Using the change of variables $ \theta=q_{t/2}^{-1}(t,s) $ 
 we eventually get 
$$
  \int_{q_{t/2}(t,x_0(t/2)-\ln 2) }^{x_0(t)}y(t,s)\, ds=\int_{x_0(t/2)-\ln 2 }^{x_0(t/2)}y(t,q_{t/2}(t,\theta))
  \partial_x q_{t/2}(t,\theta)\, d\theta
$$
  and \eqref{13aaa}-\eqref{13bbb} lead to
  \begin{equation}\label{asas}
  \int_{q_{t/2}(t,x_0(t/2)-\ln 2) }^{x_0(t)}y(t,s)\, ds \ge 
   e^{c t/8} \int_{x_0(t/2)-\ln 2 }^{x_0(t/2)} 
  y(t/2,\theta)\, d\theta
  \end{equation}
 Now,  we notice  that \eqref{tf1}  forces 
\begin{equation}\label{as}
x_0(\tau)\le x(\tau) + \ln(4/3), \quad \forall \tau\in [0,T]\; .
\end{equation}
Indeed, otherwise since $ u(\tau,x(\tau))\ge c-\frac{c}{16} $ and $ u_x(\tau,\cdot) \ge  u(\tau,\cdot)$ on $]-\infty,x_0(\tau)] $
 this would imply that $ u(\tau, x(\tau)+\ln(4/3) )\ge \frac{15}{16} c $ that is not compatible with 
 $ \varphi_c(\ln(4/3) )=\frac{4}{3}c $  and \eqref{assum4}.
From \eqref{as} we deduce that for all $\tau\in [0,T]$,
\begin{equation}\label{as2}
x_0(\tau)-\ln(2) \le x(\tau)-\ln(3/2)
\end{equation}
and thus 
$$
u(\tau,  x_0(\tau)-\ln(2))\le \varphi_c (x_0(\tau)-x(\tau)-\ln(2))+\frac{c}{16} 
\le \varphi_c (-\ln(3/2))+\frac{c}{16} \le \frac{3c}{4}\; .
$$
Combining this last inequality at $ \tau=t/2 $  with  \eqref{defq1}, \eqref{nanaa2}  and a continuity argument we infer that 
$$
\dot{x}(\tau) -\partial_t q_{t/2}(\tau, x_0(t/2)-\ln(2))\ge \frac{c}{8} , \quad \forall \tau\in [t/2,T]\; ,
$$
which yields 
\begin{equation}\label{as3}
q_{t/2}(t,x_0(t/2)-\ln(2)) \le x(t)-\ln(3/2)-\frac{c}{16} t \; .
\end{equation}
Combining \eqref{asas} and \eqref{as3}, \eqref{94} follows.
\end{proof}


\begin{corollary}\label{coro1}
Under the same hypotheses as in Proposition \ref{pro2},  for all $t\in[0,T]$, it holds 
\begin{equation}\label{lo}
u(t,\cdot)-6v(t,\cdot) \le  e^{9-\frac{ct}{32}}\Vert y_0\Vert_{L^1(\R)} \quad \text{on} \; ]r(t)-8, +\infty[ \;, 
\end{equation}
where $v=(4-\partial_x)^{-1} u $. 
\end{corollary}
\begin{proof} By \eqref{zq}, it hods 
\begin{align}
6v-u &=  (1-\partial_x^2)^{-1}y - 2 (4-\partial_x^2)^{-1}y  = \frac{1}{2} e^{-|\cdot|} \ast y - \frac{1}{2} e^{-2|\cdot|} \ast y 
\nonumber\\
&= \frac{1}{2} (e^{-|\cdot|}- e^{-2|\cdot|})\ast y \nonumber\\
& \ge -\frac{1}{2} (e^{-|\cdot|}- e^{-2|\cdot|})\ast y^-\ge -\frac{1}{2} e^{-|\cdot|} \ast y^-. \label{ql}
\end{align}
Therefore, for $ x\ge  r(t)-8 $, \eqref{50}, \eqref{81} and \eqref{choice} lead to 
\begin{eqnarray*}
6v(x)-u(x) & \ge &  - \frac{1}{2}\int_{-\infty}^{r(t)-\frac{c}{16}t} e^{-|x-z|} y^-(z)\, dz- \frac{1}{2}
\int_{r(t)-\frac{c}{16}t}^{+\infty} e^{-|x-z|} y^-(z)\, dz\\
& \ge & -\frac{1}{2} e^{0\wedge (8-\frac{c}{16}t)}  (1+2e^{2^{-5} c t})\|y_0\|_{L^1}  -\frac{1}{2} e^{-ct/8}\Vert y_0\Vert_{L^1(\R)}
\\
 & \ge & -e^{9-\frac{ct}{32}}\Vert y_0\Vert_{L^1(\R)} \; .
\end{eqnarray*}
\end{proof}
\section{Proof of Theorem \ref{stabpeakon}}\label{sect5}
Before starting the proof, we need the two following lemmas that will help us to rewrite the problem in a slightly different way. 
The next lemma ensures that the distance in $ \H $ to the translations of $\varphi_c$ is minimized for any  point of maximum of $ v=(4-\partial_x^2)^{-1} u$. 
\begin{lemma}[Quadratic Identity \cite{LL}]\label{Lemma P2} For any $u\in L^{2}(\mathbb{R})$ and $\xi\in\mathbb{R}$, it holds
\begin{equation}
E(u)-E(\varphi_{c})=\|u-\varphi_{c}(\cdot-\xi)\|^{2}_{\mathcal{H}}+4c\left(v(\xi)-\frac{c}{6}\right),
\label{P8}
\end{equation}
where $v=(4-\partial^{2}_{x})^{-1}u$ and $\xi $ is any point where $ v$ reaches its maximum.
\end{lemma}
We will also need the following lemma that is implicitly contained in \cite{AK}.
\begin{lemma}\label{LemmaPP}
Let $ u\in L^\infty(\R) \cap L^2(\R) $ such that 
\begin{equation}\label{hyg}
 \|u-\varphi_c(\cdot-r) \|_{L^\infty}\le 10^{-5}\; c 
\end{equation}
for some $ c>0 $ and some $ r\in \R $. Then $ v=(4-\partial_x^2)^{-1} u$ has got a unique point of maximum $ \xi $ on $\R $ and 
\begin{equation}\label{hyg2}
\|u-\varphi_c(\cdot-\xi) \|_{\H}\le \|u-\varphi_c(\cdot-r) \|_{\H}\; .
\end{equation}
Finally, $\xi\in \Theta_r= [r-6.7,r+6.7]$,  is the only critical point of $ v$ in $\Theta_r $ and 
\begin{equation}\label{hyg3}
 \sup_{x\not\in \Theta_r} \big(|u(x)|, |v(x)|,|v_x(x)|\big) \le \frac{c}{100} \; .
\end{equation}
\end{lemma}
\begin{proof}
Let us first recall that  $v-\rho_c= \frac{1}{4} e^{-2|\cdot|}\ast (u-\varphi_c) $ so that  Young convolution inequalities yield
\begin{equation}\label{jd}
\|v-\rho_c\|_{L^\infty}\le \|\frac{1}{4} e^{-2|\cdot|} \|_{L^1} \|u-\varphi_c\|_{L^\infty} \le \frac{1}{4}  \|u-\varphi_c\|_{L^\infty} 
\quad\text{and}\quad \|(v-\rho_c)'\|_{L^\infty}\le \frac{1}{2}  \|u-\varphi_c\|_{L^\infty} \; .
\end{equation}
Moreover,  $(v-\rho_c)''=4(v-\rho_c)-(u-\varphi_c) $ leads to
$$
\|(v-\rho_c)''\|_{L^\infty}\le 2 \|u-\varphi_c\|_{L^\infty} \; .
$$
Now, the crucial observations in \cite{AK1} are that 
\begin{equation}
\rho''\le \frac{\sqrt{2}-2}{6} \; \text{on} \; {\mathcal V}_0, \quad \rho'(x)=-\rho'(-x)\ge
 10^{-4}, \;\forall x\in \Theta_0/{\mathcal V}_0 \;, \label{zs1}
 \end{equation}
where, $\forall r>0 $,  $ {\mathcal V}_r=[r-\ln\sqrt{2},r+\ln\sqrt{2}] $. 
Therefore, \eqref{hyg} together with \eqref{zs1} ensure that $ v' $ is strictly decreasing on  $ {\mathcal V}_r $ and that 
 $ v'>0 $ on $[r-6.7,r-\ln\sqrt{2}]$ and $ v'<0 $ on $ [r+\ln\sqrt{2},r+6.7] $. This proves that $ v $ has got a unique critical point  $ \xi $ on $ \Theta_r $ that is a local maximum and that $\xi\in {\mathcal V}_r\subset \Theta_r$. Moreover $\rho(0)=1/6 $ together with the direct estimates 
\begin{equation}
  \rho\vee |\rho'|\le 5\times 10^{-4} \;\text{on} \; \R/\Theta_0\quad \text{and}\quad
  \varphi\vee |\varphi'|\le 5\times 10^{-3} \;\text{on} \; \R/\Theta_0\; ,
  \label{zs2}
  \end{equation}
  ensure that this is actually the unique point of maximum of $ v$ on $ \R $. This proves the first part of \eqref{hyg3} whereas the second part follows again from \eqref{zs2}. Finally, \eqref{hyg2} follows directly from Lemma \ref{Lemma P2} together with the fact $ v(\xi) $ is the maximum of $ v$ on $ \R $. 
\end{proof}
Now, let us recall that, by \eqref{app}, we can approximate  any $ u_0\in Y $ satisfying Hypothesis \ref{hyp}  by a sequence $ \{u_{0,n}\}\subset Y \cap H^\infty(\R) $ satisfying Hypothesis \ref{hyp}  such that 
$$
u_{0,n} \to u_0 \;  \text{in} \;  H^1(\R) \cap W^{1,1}(\R) \quad  \text{and}\quad   \|y_n\|_{L^1} \le \|y\|_{\M}\; , \; 
\forall n\in\N.
$$
Therefore the continuity with respect to initial data in Proposition  \ref{weakGWP} ensures that  to prove Theorem  \ref{stabpeakon} we can reduce ourself to  initial data $ u_0\in Y\cap H^\infty $. 

Let $ \eps_0 $ be the universal constant defined in \eqref{choice} and let us fix
\begin{equation}\label{km}
 0<\eps<\eps_0 \frac{1\wedge c^2}{(2+c)^3} \; .
\end{equation}
Let us also fix $ A>0 $. 
 From the continuity with respect to initial data \eqref{cont2} at 
 $ \varphi_c $, the fact that $t\mapsto \varphi_c(\cdot-ct)$ is an exact solution   and the translation symmetry of the (D-P) equation, there exists 
 \begin{equation}\label{defdelta}
  0<\delta=\delta
(A,\varepsilon,c)\le \eps^4
\end{equation}
 such that for any $ u_0\in Y $ satisfying Hypothesis \ref{hyp} and \eqref{P1}-\eqref{P2}
 with $ A $ and $\delta $, it holds 
\begin{equation}\label{kj}
\|u(t)-\varphi_c(x-ct) \|_{\H} \le 2 (2+c)\;\varepsilon , \quad \forall t\in [0, T_{\eps}], \text{with} \quad 
T_\eps=\max\Bigl( 0,\frac{32}{c} \big(9+\ln (A/\eps^2)\big)\Bigr)\; ,
\end{equation}
where $u\in C(\R_+; H^1(\R)) $ is the solution of the (D-P) equation emanating from $ u_0$. So let $ u_0\in Y\cap H^\infty(\R) $ that satisfies Hypothesis \ref{hyp}  and \eqref{P1}-\eqref{P2}
 with $ A $ and $\delta $. \eqref{kj} together with the definition \eqref{choice}
 of $\eps_0 $  and   Lemma \ref{LemmaPP1} then ensure that
 \begin{equation}\label{kjjj}
\|u(t)-\varphi_c(x- ct) \|_{L^\infty} <10^{-5} c  , \quad \forall t\in [0, T_{\eps}], 
\end{equation}
 and Lemma \ref{LemmaPP}   then ensures that
\begin{equation}\label{kjj4}
\|u(t)-\varphi_c(x-\xi(t)) \|_{\H} \le  2 (2+c) \;\varepsilon \; , \quad \forall t\in [0, T_{\eps}], 
\end{equation}
 where $ \xi(t) $ is  the only  point where $ v(t)=(4-\partial_x^2)^{-1} u(t)$ reaches its maximum.

By a continuity argument it remains to prove that  for any $ T\ge T_\eps $, if 
\begin{equation}\label{kkk}
\inf_{r\in\R} \|u(t)-\varphi_c(x-r) \|_{\H} \le  3  (2+c) \;\varepsilon  \quad\text{on} \quad [0,T] 
\end{equation}
then  $   v(T)=(4-\partial_x^2)^{-1} u(T)$ reaches its maximum on $ \R $ at a unique point $ \xi(T) $ 
and 
\begin{equation}\label{kkkk}
\|u(T)-\varphi_c(x-\xi(T)) \|_{\H} \le  2 (2+c) \;\varepsilon \; .
\end{equation}
At this stage it is worth noticing that, as above, \eqref{kkk} together with the definition \eqref{choice}
 of $\eps_0 $  and   Lemma \ref{LemmaPP1} ensure that
 $$
 \inf_{r\in\R} \|u(t)-\varphi_c(x-r) \|_{L^\infty} \le 10^{-5} c   , \quad \forall t\in [0, T] \; .
 $$
 Therefore applying Lemma \ref{LemmaPP}  and again   Lemma \ref{LemmaPP1} we obtain that 
\begin{equation}\label{kjj}
\|u(t)-\varphi_c(x-\xi(t)) \|_{\H} \le  3 (2+c) \;\varepsilon \; 
\quad \text{and} \quad  \|u(t)-\varphi_c(x-\xi(t)) \|_{L^\infty} \le 10^{-5} c  ,\quad \forall t\in [0, T], 
\end{equation}
 where $ \xi(t) $ is  the only  point where $ v(t)=(4-\partial_x^2)^{-1} u(t)$ reaches its maximum. Moreover, \eqref{kkk}   together with \eqref{km}, \eqref{choice},  Corollary \ref{coro1} and the definition of $T_\eps $ in \eqref{kj} then ensure that 
\begin{equation}\label{kj2}
u(t,\cdot)-6v(t,\cdot) \le  \eps^2 \quad \text{on} \quad \Theta_{\xi(t)} , \quad \forall t\in [0,T] \; .
\end{equation}
To prove \eqref{kkkk}, we  follow closely the proof in \cite{AK}, keeping \eqref{kj2} in hands. 
The idea comes back to \cite{CS}  and consists in constructing two functions $ g$ and $ h$ that permits to link in a good way 
 $ E(u)$, $F(u) $ and the maximum $M $ of $ v=(4-\partial^{2}_{x})^{-1}u $. This was first implement in \cite{LL} for the (DP)-equation under the additional hypothesis that the   momentum density of the initial data is non negative.

\begin{lemma}[ See \cite{LL}]\label{Lemma P4}
Let $u\in L^{2}(\mathbb{R})$ and $v=(4-\partial^{2}_{x})^{-1}u\in H^{2}(\mathbb{R})$. Denote by $M=\max_{\R}v(\cdot)=v(\xi)$ and define the function $g$ by
\begin{equation}
  g(x)=\left\{
    \begin{aligned}
     &2v(x)+v_{xx}(x)-3v_{x}(x)=u(x)-6v_x(x)+12v(x),\qquad \forall x\le \xi,\\
     &2v(x)+v_{xx}(x)+3v_{x}(x)=u(x)+6v_x(x)+12v(x),\qquad \forall x\ge \xi.\\
    \end{aligned}
  \right.
  \label{GG1}
\end{equation}
Then it holds
\begin{equation}
\int_{\mathbb{R}}g^{2}(x)dx=E(u)-12M^{2},
\label{GG2}
\end{equation}
and
\begin{equation}\label{improvment}
\int_{\mathbb{R}}g^{2}(x)dx=\big\Vert u-\varphi_c(\cdot-\xi)\big\Vert_{\mathcal{H}}^2-12\left(\frac{c}{6}-M\right)^2
\le \big\Vert u-\varphi_c(\cdot-\xi)\big\Vert_{\mathcal{H}}^2\; .
\end{equation}

\end{lemma}
\begin{proof} The first identity is proven in \cite{LL} by combining  integration by parts and the fact that 
 $ v_x(\xi)=0 $.
To prove the second identity we remark that by the definition of $\rho_c$ in (\ref{1.88}), it holds 
\begin{equation*}
  \left\{
    \begin{aligned}
     &2\rho_c(\cdot-\xi)-\rho''_c(\cdot-\xi)+3\rho'_c(\cdot-\xi)=0,\qquad \forall x\le \xi,\\
     &2\rho_c(\cdot-\xi)-\rho''_c(\cdot-\xi)-3\rho'_c(\cdot-\xi)=0,\qquad \forall x\ge \xi.\\
    \end{aligned}
  \right.
  \end{equation*}
Therefore, setting $w=v-\rho_c(\cdot-\xi)=(4-\partial_x^2)^{-1}[u-\varphi_c(\cdot-\xi)]$ one may rewrite $g$ as
\begin{equation}
  g=\left\{
    \begin{aligned}
     &2w+w_{xx}-3w_{x}~~\text{on} \; ]-\infty,\xi],\\
     &2w+w_{xx}+3w_{x}~~\text{on} \; [\xi,+\infty[\\
    \end{aligned}
  \right.
  \label{GG4}
\end{equation}
and \eqref{improvment} follows by applying \eqref{GG2} with $ u $ replaced by $ u-\varphi_c(\cdot-\xi)$.
\end{proof}
\begin{lemma}[See \cite{LL}]\label{Lemma P5}
Let $u\in L^{2}(\mathbb{R})$ and $v=(4-\partial^{2}_{x})^{-1}u\in H^{2}(\mathbb{R})$. Denote by $M=\max_{\R}v(\cdot)=v(\xi)$ and define the function $h$ by
\begin{equation}
  h(x)=\left\{
    \begin{aligned}
     &-v_{xx}(x)-6v_{x}(x)+16v(x),\qquad \forall x\le \xi.\\
     &-v_{xx}(x)+6v_{x}(x)+16v(x),\qquad \forall x\ge \xi.\\
    \end{aligned}
  \right.
  \label{HH1}
\end{equation}
Then, it holds
\begin{equation}
F(u)-144M^{3}=\int_{\mathbb{R}}h(x)g^{2}(x)dx.
\label{HH2}
\end{equation}
\end{lemma}
Gathering Lemmas \ref{Lemma P2}, \ref{Lemma P4} and \ref{Lemma P5} and making use of  \eqref{kj2} we derive the crucial 
  relation that linked  $ E(u)$, $F(u) $ and the maximum $M $ of $ v=(4-\partial^{2}_{x})^{-1}u $.

\begin{lemma}\label{Lemma P6}
Let $ \eps>0 $ and $u\in L^{2}(\mathbb{R})$   be such that $v=(4-\partial_x^2)^{-1} u $ has got a unique point $ \xi $ of maximum on $ \R $ with 
\begin{equation}\label{assum6}
\Vert u-\varphi_c(\cdot-\xi)\Vert_{\H}\le 3(2+\eps) \eps, \quad  \Vert u-\varphi_c(\cdot-\xi)\Vert_{L^\infty}\leq   10^{-5} c  \quad  \; \text{and} \; 
u-6v\le  \eps^2 \quad \text{on} \quad  \Theta_\xi\quad \; .
\end{equation}
Then, setting $ M= v(\xi) $, it holds 
 \begin{equation}
M^{3}-\frac{1}{4}E(u)M+\frac{1}{72}F(u)\leq \frac{(2+c)^2}{8} \eps^4\; .
\label{P33}
\end{equation}
\end{lemma}

\begin{proof}
The key is to show that the function $ h $ defined in Lemma \ref{Lemma P5} satisfies $h\leq 18M+\eps^2$ on $\R$. We notice that $ h $ may be rewritten as 
$$
  h(x)=\left\{
    \begin{aligned}
     &u(x)-6v_{x}(x)+12v(x),\qquad \forall x\le \xi.\\
     &u(x)+6v_{x}(x)+12v(x),\qquad \forall x\ge \xi.\\
    \end{aligned}
  \right.
  $$
and  that \eqref{jd} together with  the second inequality in \eqref{assum6} force
 \begin{equation}\label{nz}
|M-c/6|\le 10^{-5} c \; .
 \end{equation}
 Moreover,  Lemma \ref{LemmaPP}  ensures that $ v_x >0 $ on $]\xi-6.7,\xi[ $ and 
$ v_x<0 $ on $ ]\xi, \xi+6.7[ $. 

 We divide $ \R $ into three intervals.  For $ x\in\R/\Theta_{\xi}$, \eqref{hyg3} with $ r=\xi  $  and then \eqref{nz}  ensure that 
\begin{align}
h(x)&\le |u(x)|+6\vert v_x(x)\vert+12|v(x)| \le \frac{19 c}{100}\le 18M.
\label{hles18m}
\end{align}
For  $\xi-6.7<x<\xi$, then $v_x\geq0$ and using that $u-6v\leq\eps^2$ on $\Theta_\xi $, we get 
$$
h(x)\leq18M+\eps^2.
$$
If $\xi<x<\xi+6.7$, then $v_x\leq0$ and using that $u-6v\leq\eps^2$ on $\Theta_\xi $, we get 
$$
h(x)\leq18M+\eps^2.
$$
Therefore it holds, 
$$
h\leq18M+\eps^2 \quad \text{on} \quad \R\; .
$$
Combining \eqref{GG2}, \eqref{improvment}, \eqref{HH2} and the first inequality in \eqref{assum6}, one eventually gets
\begin{align*}
F(u)-144M^3 & =\int_{\R}h(x)g^2(x)dx\leq 18M\big(E(u)-12M^2\big)+\varepsilon^2 \|u-\varphi_c(\cdot-\xi)\|_{\H}^2\\
& \le 18 M E(u)-72M^3 + 9(2+c)^2 \varepsilon^4 \; .
\end{align*}
that completes the proof of the lemma.
\end{proof}
Finally, we will need the following lemma that links the distance between $F(u_0)$ and $F(\varphi_c) $ to  the distance between $ u_0 $ and $ \varphi_c $ in $L^2(\R) $.
\begin{lemma}\label{Lemma P1}
Let $u_0\in Y$ that satisfies Hypothesis \ref{hyp}. If for some $ 0<\gamma<1 $ it holds 
$$
\|u_0-\varphi_{c}\|_{\mathcal{H}}\le \gamma
$$
then 
\begin{equation}
|E(u_0)-E(\varphi_{c})|\le 2 \gamma (2+c)
\label{P4}
\end{equation}
and 
\begin{equation}
|F(u_0)-F(\varphi_{c})|\le 6 \gamma (2+c)^2 \;, 
\label{P5}
\end{equation}
where $\varphi_c$, $\rho_c$ are defined  in (\ref{1.7}), (\ref{1.88}).
\end{lemma}
\begin{proof}
By the triangle inequality and \eqref{EE},
\begin{align*}
|E(u_0)-E(\varphi_{c})|& \le \|u_0-\varphi_c\|_{\H} (\|u_0\|_{\H}+\|\varphi_c\|_{\H})\\
& \le \|u_0-\varphi_c\|_{\H} ( \|u_0-\varphi_c\|_{\H} +2\|\varphi_c\|_{\H})\\
& \le \gamma (\gamma+\frac{2}{\sqrt{3}} c ) 
\end{align*}
Now, 
\begin{align*}
|F(u_0)-F(\varphi_{c})|& \le \|u_0-\varphi_c\|_{L^2} \Bigl\| u^2+2u \varphi_c +\varphi_c^2\Bigr\|_{L^2}\\
& \le \|u_0-\varphi_c\|_{L^2} 
\Bigl[  \|u_0\|_{L^\infty} \|u_0\|_{L^2} + \| \varphi_c\|_{L^\infty} (2 \|u_0\|_{L^2} + \|\varphi_c\|_{L^2} )\Bigr] \\
\end{align*}
and Lemma \ref{LemmaPP1} together with \eqref{1.5} then yield 
$$
|F(u_0)-F(\varphi_{c})|\le 2\gamma \Bigl[4 \sqrt{\gamma} (2+c) +c(4\gamma+3 c)\Bigr]\; .
$$
\end{proof}

According   to  \eqref{kjj}-\eqref{kj2}  and  Lemma \ref{Lemma P6}, setting $M=v(\xi(T)) $, we get 
$$
M^{3}-\frac{1}{4}E(u)M+\frac{1}{72}F(u)\leq \frac{(2+c)^2\eps^4}{8}\; .
$$
The conservation of $ E$ and $ F $ together with Lemma \ref{Lemma P1} and \eqref{defdelta} then lead to
\begin{eqnarray}
M^{3}-\frac{1}{4}E(\varphi_{c})M+\frac{1}{72}F(\varphi_{c}) & \leq& \frac{1}{4}\vert E(u_0)-E(\varphi_c)\vert+\frac{1}{72}\vert F(u_0)-F(\varphi_c)\vert+\frac{(2+c)^2 \eps^4}{8}\nonumber \\
 & \le &\eps^4 (2+c)^2
\label{55;}
\end{eqnarray}
Now, by \eqref{EE} and (\ref{i}) one can check that $E(\varphi_{c})=c^{2}/3$ and $F(\varphi_{c})=2c^{3}/3$, 
 so that  (\ref{55;}) becomes
$$
\left(\frac{c}{6}-M\right)^2\left(M+\frac{c}{3}\right)  \le \eps^4 (2+c)^2\; .
$$
Finally, since according to \eqref{nz} $ M\ge 0 $, we deduce that 
$$
\big|\frac{c}{6}-M\big| \le \sqrt{\frac{3}{c}} \; (2+c) \eps^2
$$
which together with  Lemma \ref{Lemma P2} , Lemma \ref{Lemma P1}  and \eqref{defdelta} ensure that 
$$
\|u(T)-\varphi_c(x-\xi(T)) \|_{\H}^2 \le \eps^2 \Bigl( 4  \sqrt{3c}(2+c) + 2(2+c)\eps^2\Bigr)
\le 4 (2+c)^2 \eps^2  \; .
$$
This completes the proof of  \eqref{kkkk} and thus of  \eqref{P3}. Note that \eqref{P44} then follows by using Lemma \ref{LemmaPP1}.

\section{Stability of a train of well-ordered antipeakons-peakons}\label{sec-multi}

In this section, we generalize the stability result to the sum of well ordered trains of antipeakons-peakons (see fig \ref{twowellodantipeakon} and fig \ref{wellodantipeakon}).
\begin{figure}[!htb]
\vspace{0cm}
\centering
\subfloat[
 Two antipeakons at speeds $c_i=1,4$.]
{\includegraphics[width=8cm, height=6.5cm]{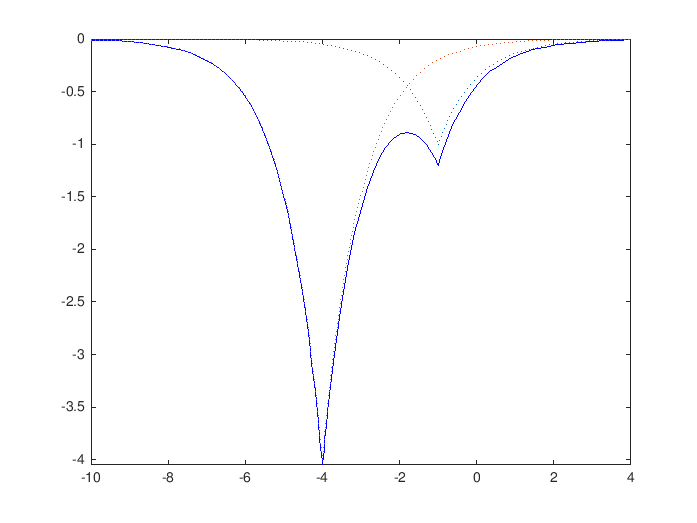}\label{wanitpeakonnnnn}}
\subfloat[
 Two peakons at speeds $c_i=2,4$.]
{\includegraphics[width=8cm, height=6.5cm]{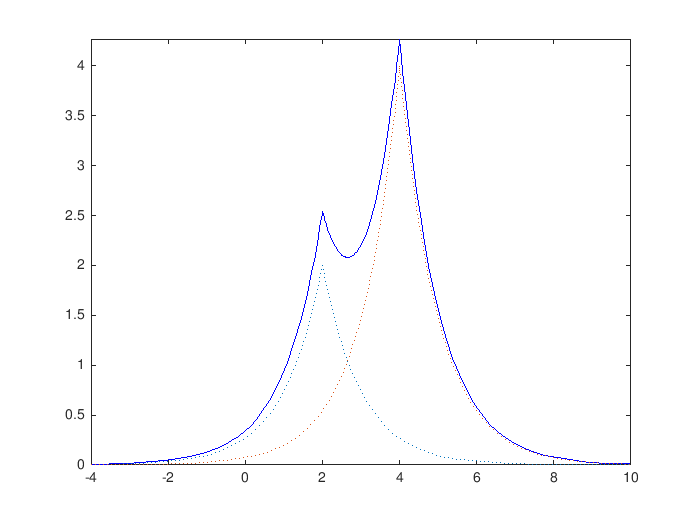}\label{wellantipeakonnnnn}}
\caption{Summing two antipeakons and peakons profiles at time $t=1$ with different speeds. }
\label{twowellodantipeakon}
\end{figure}
\begin{figure}[!htb]
\vspace{0cm}
\centering
\subfloat[
 At time $t=1$.]
{\includegraphics[width=8cm, height=6.5cm]{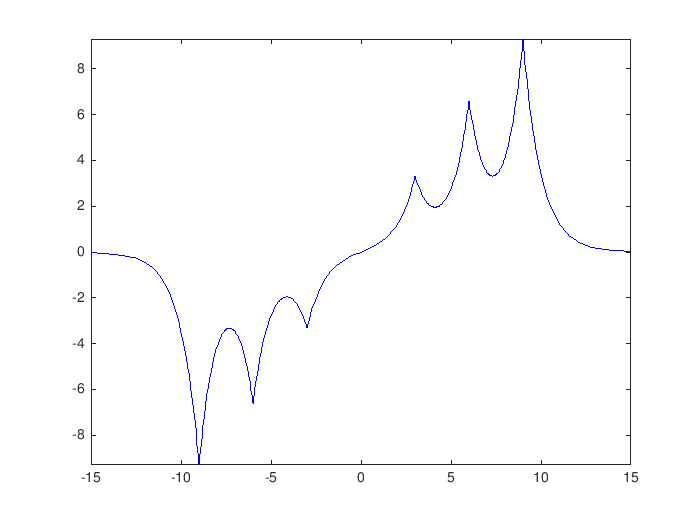}\label{wanitpeakon}}
\subfloat[
 At time $t=3$.]
{\includegraphics[width=8cm, height=6.5cm]{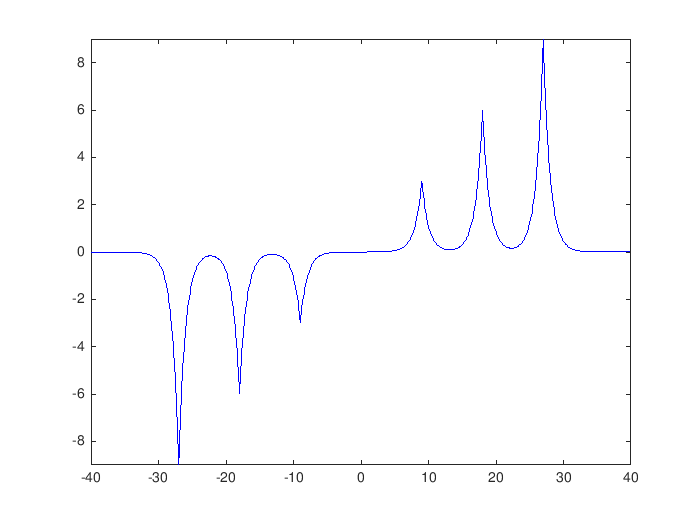}\label{wellantipeakonn}}
\caption{Three well-ordered trains of antipeakons and peakons profiles at different speeds $c_i= 3,6,9$. }
\label{wellodantipeakon}
\end{figure}
Let be given $ N_- + N_+ $  ordered  speeds $\vec{c}=(c_{-N_-},..,c_{-1},c_1,..,c_{N_+} )\in \R^{N_-+N_+} $ with 
\begin{equation} \label{speed1}
c_{-N_-}<..<c_{-1}<0 <c_1<..<c_{N_+} \; .
\end{equation}
We set 
\begin{equation} \label{speed2}
\|\vec{c}\|_1=\sum_{j=-N_-}^{N_+}  |c_j|\quad\text{and}\quad \sigma(\vec{c})=\min_{i\in
 [[1-N_-,N_+]]} |c_i-c_{i-1}| \; 
\end{equation}
where to simplify the notations we set 
\begin{equation} \label{speed0}
c_{0}=0 \; .
\end{equation}
For $ \alpha>0 $ and $ L>0 $ and $ \vec{c}$ satisfying \eqref{speed1}-\eqref{speed2}, we define the following neighborhood of all the sums
  of $N_-+N_+$ well-ordered antipeakons and peakons of speed $ c_{-N_-},..,c_{-1},c_1,..,c_{N_+} $ with spatial shifts $ z_j $ that satisfied
$ z_j-z_{q}\ge L $ for $ j>q$. 
  \begin{equation}
U(\alpha,L,\vec{c}) = \Bigl\{
u\in L^2(\R), \, \inf_{z_j-z_{q}> L, \; j>q} \big\|u-\sum_{j=-N_-\atop j\neq 0}^{N_+}\varphi_{c_j} (\cdot-z_j) \big\|_{\mathcal{H}} < \alpha \Bigr\}\; .
\end{equation}

We start by establishing the following lemma that linked the distance in $ L^\infty $ to the train of antipeakons-peakons 
 with the distance in $ \H $.
Indeed, applying  Lemma \ref{LemmaPP1} with $ \psi=\sum_{j=-N_-\atop j\neq 0}^{N_+} \varphi_{c_j}(\cdot -z_j) $ and 
 observing that 
 $$
 \|\psi \|_{L^\infty}+\|\psi'\|_{L^\infty} \le 2\sum_{j=-N_-\atop j\neq 0}^{N_+}\|\varphi_{c_j}\|_{L^\infty}\le 2 
  \|\vec{c} \|_{1} \; ,
 $$
 we get the following lemma.
\begin{lemma}[$L^{\infty}$  approximations]\label{LemmaPP1loc}
Let $(c_j, z_j) \in\R^2$, $j\in [[N_-,N_+]]\setminus\{0\} $, and $u\in Y $, satisfying Hypothesis 1, then  
\begin{equation}
\Big\|u-\sum_{j=-N_-\atop j\neq 0}^{N_+}\varphi_{c_j}(\cdot-z_j)\Big\|_{L^{\infty}(\mathbb{R})}\le 2 \Big\|u-\sum_{j=-N_-\atop j\neq 0}^{N_+}\varphi_{c_j}(\cdot-z_j)\Big\|_{\mathcal{H}}^{2/3} \Big(1+  \sqrt{2} \Big\|u-\sum_{j=-N_-\atop j\neq 0}^{N_+}\varphi_{c_j}(\cdot-z_j)\Big\|_{\mathcal{H}}^{2/3}+ 2 \|\vec{c} \|_1 \Big) \; .
\label{PP2loc}
\end{equation}
In particular, if moreover $ \|u-\sum_{j=-N_-\atop j\neq 0}^{N_+}\varphi_{c_j}(\cdot-z_j)\|_{\mathcal{H}}
\le 1/2 $ then 
\begin{equation}
\Big\|u-\sum_{j=-N_-\atop j\neq 0}^{N_+}\varphi_{c_j}(\cdot-z_j)\Big\|_{L^{\infty}(\mathbb{R})}\le 
4(1+\|\vec{c}\|_1)  \Big\|u-\sum_{j=-N_-\atop j\neq 0}^{N_+}\varphi_{c_j}(\cdot-z_j)\Big\|_{\mathcal{H}}^{2/3} \; .
\label{PP3loc}
\end{equation}
\end{lemma}
\noindent


\subsection{Control of the distance between the peakons}
 In this subsection we want to prove that for a given $ \vec{c}$  satisfying \eqref{speed1} , there exists 
  $ \alpha=\alpha(\vec{c})$ and $ L=L(\vec{c}) $ such that as soon as the solution $ u(t) $ stays in 
   $ U(\alpha,L,\vec{c}) $  the different bumps of $ u $   that are individually close to a peakon or an antipeakon  get away  from each others as time is increasing.
 This is crucial in our analysis since we do not know how to manage strong interactions. 

\begin{lemma}\label{eloignement}(Decomposition of the solution around a sum of antipeakons and peakons).
Let $ u_0 \in  Y $ satisfying \eqref{huhu0}-\eqref{distz}. There exist $\alpha_0(\vec{c})>0$,
$L_0(\vec{c})>0$ and $\tilde{K}(\vec{c})\geq1$  such that for all  $0<L_0<L$ if for some $T>0$
\begin{equation}\label{hyplem}
u\in U(\alpha_0, L/2,\vec{c}) \quad \text{ on  $[0,T]$ }
\end{equation}
then there exist $ N_-+N_+ $  $C^1$-functions $ x_{-N_-}(\cdot)< ..<x_{-1}(\cdot)< x_1(\cdot)< ..<x_{N_+}(\cdot) $ defined on $[0,T] $ such that for all $ t\in  [0,T] $ we have,
\begin{equation}\label{orthocondition}
\int_{\R}\Big(v(t,x)-\sum_{j=-N_-\atop j\neq 0}^{N_+}\rho_{c_j}\big(x-x_j(t)\big)\Big)\partial_x\rho_{c_i}\big(x-x_i(t)\big)dx=0,\qquad \forall  i\in [[-N_- , N_+]],
\end{equation}
\begin{equation}\label{149}
\vert \dot{x}_i(t)-c_i\vert\leq \frac{\sigma(\vec{c})}{8},\qquad \forall i\in [[-N_- , N_+]]\backslash \{0\},
\end{equation}
and \begin{equation}
x_i(t)-x_{j}(t) \geq 3L/4, \quad \forall (i,j)\in ([-N_-,N_+]]\setminus\{0\})^2 , i> j ,
   \label{eloi}
\end{equation}
where $ v=(4-\partial_x^2)^{-1} u $ and $\rho_{c_i}=(4-\partial_x^2)^{-1}\varphi_{c_i} $. \\
Moreover,  if 
\begin{equation}\label{1488}
u\in U(\alpha, L/2,\vec{c}) \quad \text{ on  $[0,t_0]$ }
\end{equation}
for some $0<\alpha<\alpha_0(\vec{c}),$ then
\begin{equation}\label{147}
\left\Vert u(t,\cdot)-\sum_{i=-N_-\atop i\neq 0}^{N_+}\varphi_{c_i}(\cdot-x_i(t))\right\Vert_{\mathcal{H}}\leq \tilde{K}\alpha,
\end{equation}

\begin{equation}\label{148}
\left\Vert v(t,\cdot)-\sum_{i=-N_-\atop i\neq 0}^{N_+}\rho_{c_i}(\cdot-x_i(t))\right\Vert_{C^{1}(\R)}\leq \tilde{K}\alpha,
\end{equation}

\end{lemma}
\begin{proof} The strategy is to use a modulation argument to  construct $ N_-+N_+ $ $C^1$-functions
 $ t\mapsto x_i(t)$, $ i\in [[-N_-, N_+]] \backslash \{0\}$ on $ [0,T] $ satisfying the orthogonality
 conditions \eqref{orthocondition}. The proofs of  the above estimates are  direct adaptations of  similar estimates proved in Lemma \ref{modulation}. We refer to  \cite{AK,dikalm} for details. 
 \end{proof}


\subsection{Monotonicity property}\label{Sectmonotonie}
Thanks to the preceding lemma, for $\alpha_0> 0 $ small enough and $ L_0>0 $ large enough, one can construct  $C^1$-functions $ x_{-N_-}(\cdot)<...<x_{N_+}(\cdot)$ defined on $
[0,T] $ such that \eqref{147}, \eqref{148}, \eqref{149} are satisfied.
In this subsection we state the almost monotonicity of  functionals that correspond to  the part of the functional  $E(\cdot)-\lambda F(\cdot)$  at the right  of  a curve that travels slightly at the left of the $ i $th bump of $ u $. 
To control the growth of the mass of $ y(t) $ we will also need an almost monotonicity result on $E(\cdot)+\gamma M(\cdot) $ at the right of a curve that travels slightly at the left of  the smallest positive  bump of $u$. 
As in \cite{MMT}, we introduce the $ C^\infty $-function $ \Psi $ defined  on $ \R $ by 
\begin{equation}\label{defPsi}
\Psi(x) =\frac{2}{\pi} \arctan \Bigl( \exp(x/6)\Bigr) 
\end{equation}
\begin{figure}[!htb]
\vspace{-0.5cm}
\centering
\subfloat[
 Representative curve of $\Psi$]
{\includegraphics[width=7.5cm, height=6cm]{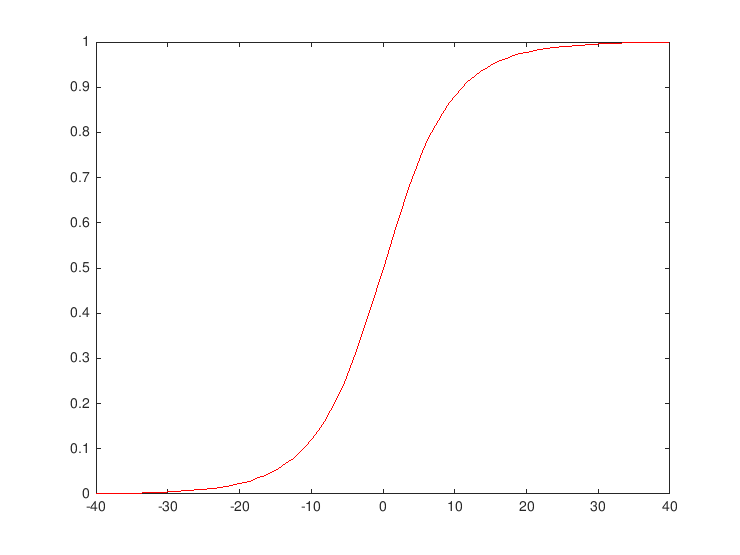}\label{psi}}
\subfloat[
 Representative curves of $\Psi'$,$\Psi''$,$\Psi'''$]
{\includegraphics[width=7.5cm, height=6cm]{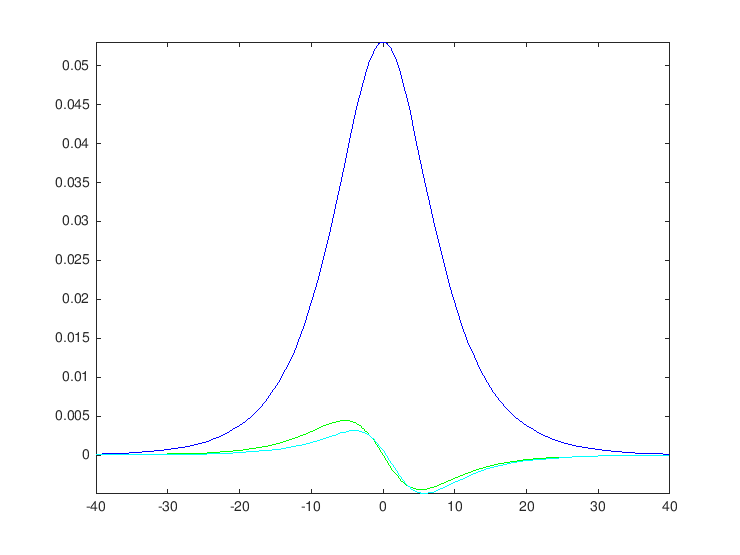}\label{psi'}}
\caption{Profiles of $\Psi$ and its derivatives.}
\end{figure}
It is easy to check that $ \Psi(-\cdot)=1-\Psi $ on $ \R $, $ \Psi' $ is a  positive even  function and that 
 there exists $C>0 $ such that $ \forall x\le 0 $, 
\begin{equation}\label{psipsi}
|\Psi(x)| + |\Psi'(x)|\le C \exp(x/6) \; .
\end{equation}
Moreover, by direct calculations (see fig \ref{psi'''psi'}), it is easy to check that 
\begin{equation}\label{psi3}
|\Psi^{'''}| \le  \frac{1}{2} \Psi' \;\text{on }\R ,
\end{equation}
\begin{figure}[!htb]
\vspace{-0.2cm}
\centering
{\includegraphics[width=8.5cm, height=6.5cm]{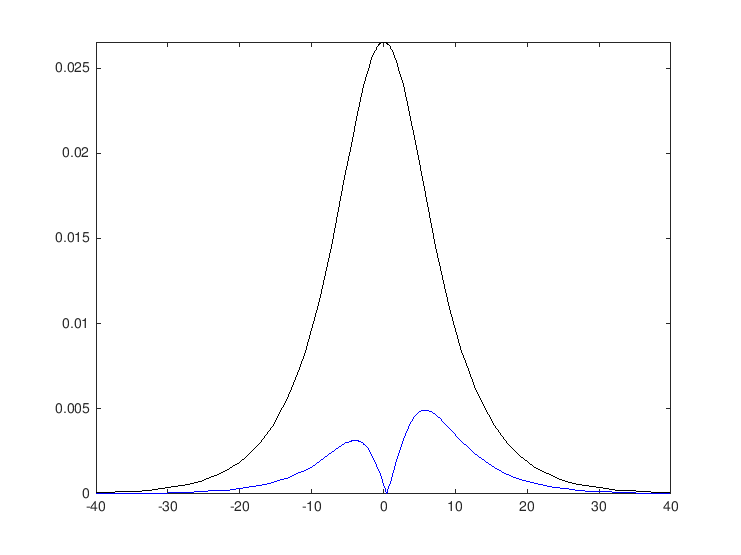}}
\caption{Profiles of $\vert \Psi''' \vert$ (blue) with respect to $\frac{1}{2}\Psi '$ (black).}
\label{psi'''psi'}
\end{figure}
and that 
\begin{equation} \label{po}
\Psi'(x)\ge \Psi'(2)= \frac{1}{3\pi} \frac{e^{1/3}}{1+e^{2/3}} , \quad \forall x\in [0,2] \; .
\end{equation}
 Setting $ \Psi_K=\Psi(\cdot/K) $, we introduce for $ j\in \{1,..,N_+\}
 $ and $ \lambda\ge 0 $,
 \begin{equation}\label{defjk}
 \mathcal{J}_{j,\lambda}(t)=\mathcal{J}_{j,\lambda,K}\big(t,u(t,x)\big)=  \int_{\R}\Bigl(  \big[4v^{2}(t,x)+5v^{2}_{x}(t,x)+v^{2}_{xx}(t,x)\big] -\lambda u^3(t,x)\Bigr) \Psi_{j,K}(t) \, dx\,,
 \end{equation}
 where $ \Psi_{j,K}(t,x)=\Psi_K(x-y_j(t)) $ with $ y_j(t)$,
 $ j=1,..,N_+ $, defined by
  \begin{equation}\label{defy1}
y_{1}(t)=x_{1}(0)+\frac{c_{1}}{2} t -\frac{L}{4},
  \end{equation}
 and
    \begin{equation}\label{defyi}
y_i(t)=\frac{ x_{i-1}(t)+ x_i(t)}{2},\quad
  i=2,..,N_+.
  \end{equation} 
   \begin{proposition}\label{monotonicitylem}(Almost monotony  of the functional energy $\mathcal{J}_{i,\lambda,K}$)
  Let $ T>0 $ and $ u\in C(\R_+;H^1) $ be the  solution of the DP equation emanating from $ u_0\in Y $, satisfying Hypothesis \ref{hyp} with \eqref{huhu0}-\eqref{huhu}  on $[0,T] $. There exist $ \alpha_0(\vec{c})>0 $ and $ L_0(\vec{c})>0 $  such that
    if $0<\alpha<\alpha_0(\vec{c})\ll1$ and $ L\ge L_0>0 $ then for any $ 1\le K \lesssim \sqrt{L}$
     and $ 0\le \lambda\le  \frac{1}{2c_1}$,
    \begin{equation}\label{monotonicityestim}
    \mathcal{J}_{j,\lambda,K}(t)-\mathcal{J}_{j,\lambda,K}(0)\le O( e^{-\frac{ L }{48K}}) ,
    \quad \forall j\in\{1,...,N_+\}, \; \quad \forall t\in [0,T] \; .
    \end{equation}
\end{proposition}
The proof of this proposition relies on the following virial type identities that are proven in the appendix.
\begin{lemma}\label{lemvir}(Viral type identity). Let $u\in C(\R_+; H^\infty(\R))$ be a solution of equation \eqref{DP}. For any smooth function $g\colon\R\mapsto\R$, it holds
\begin{multline}\label{163}
\frac{d}{dt}\int_{\R}(4v^2+5v_x^2+v_{xx}^2)(t,x)gdx=\frac{2}{3}\int_{\R}u^3(t,x)g'dx-4\int_{\R}u^2(t,x)v(t,x)g'dx\\
+5 \int_{\R} v(t,x)h(t,x) g'dx  +\int_{\R} v_x(t) h_x(t,x) g' dx,
\end{multline}
\begin{equation}\label{164}
\frac{d}{dt}\int_{\R}u^3(t,x)gdx=\frac{3}{4}\int_{\R}u^4(t,x)g'dx+\frac{9}{4}\int_{\R}(h^2-h_x^2)(t,x)g'dx,
\end{equation}
and 
 \begin{equation}
 \frac{d}{dt}\int_{\R}  y g \, dx =  \int_{\R} y u g' \, dx +\frac{3}{2} \int_{\R} (u^2-u_x^2) g'  \, dx \label{go2}
 \end{equation}

where $y=(1-\partial_x^2)u$, $v=(4-\partial_x^2)^{-1}u$, and $h=(1-\partial_x^2)^{-1}u^2$.
\end{lemma}
\noindent
{\bf Proof of Proposition \ref{monotonicitylem}}
We first note that combining \eqref{defyi} and \eqref{149}, it holds for $j=1,...,N_+$,
\begin{equation}\label{166}
\frac{3}{2}c_{N_+}\geq \dot{y}_j(t)\geq\frac{c_{1}}{2}.
\end{equation}
Now,  using \eqref{defjk}, \eqref{163} and \eqref{164} with $g=\Psi_{j,\lambda,K}(\cdot-y_j(t))$, $j\geq 1$, one gets
\begin{align}
\frac{d}{dt}\mathcal{J}_{j,\lambda,K}(t)&=-\dot{y}_j(t)\int_{\R}(4v^2+5v_x^2+v_{xx}^2)\Psi_{j,K}'(x-y_j(t))dx+\int_{\R}
(\frac{2}{3} u-4v) u^2 \Psi_{j,K}'dx\\
&+\int_{\R}(5vh+v_x h_x)\Psi_{j,K}'dx+\lambda\dot{y}_j(t)\int_{\R}u^3\Psi_{j,K}'dx
\nonumber\\
&-\frac{3}{4}\lambda\int_{\R}u^4\Psi_{j,K}'dx-\frac{9}{4}\lambda\int_{\R}(h^2-h_x^2)\Psi_{j,K}'dx \nonumber\\
&=-\dot{y}_j(t)\int_{\R}(4v^2+5v_x^2+v_{xx}^2)\Psi_{j,K}'(x-y_j(t))dx+F_1+F_2+...+F_{5}.\label{166b}
\end{align}
We claim that for $k=1,2,3$, it holds
\begin{equation}\label{168}
F_k\leq\frac{c_{1}}{10}\int_{\R}(4v^2+5v_x^2+v_{xx}^2)\Psi_{j,K}'(x-y_j(t))dx+\frac{C}{K}e^{-\frac{1}{6 K}(\sigma(\vec{c})t+L/8)}.
\end{equation}
For all $ t\in [0,T] $ and each $ j \in  [[1,N_+]]$  divide $ \R $ into two regions $ D_j=D_j(t) $ and $ D_j^c $ with
 $$
 D_j(t)=[ x_{j-1}(t) +L/4, x_j(t) -L/4] \quad \text{for} \quad  j\in  [[2,N_+]]\quad \text{and} \quad 
 D_1(t)=[x_1(0)-L/2,x_1(t)-L/4] \; .
 $$
 First, in view of  \eqref{149} and \eqref{defy1}-\eqref{defyi}, one can check that for $ x\in D_j^c(t)  $, we have
 \begin{equation}\label{169}
 |x-y_j(t)| \geq\sigma(\vec{c})\, t+L/8 ,
 \end{equation}
 with $ \sigma(\vec{c}) $ defined in \eqref{speed2}.
 Indeed, for $ j\in [[2,N_+]] $ it holds 
 $$ 
  |x-y_j(t)| \ge \frac{ x_{j}(t)- x_{j-1}(t)}{2}-L/4 \ge \frac{c_j-c_{j-1}}{4} \, t +L/8\geq\sigma(\vec{c})t+L/8 .
  $$
  and for $ j= 1$, 
  $$
   |x-y_1(t)| \ge \frac{c_1}{4} t + L/4 \geq\sigma(\vec{c})t+L/8 .
  $$
 Second, noticing  that 
 \begin{equation}\label{170}
 u^2=(4v-v_{xx})^2\leq20v^2+5v_{xx}^2\leq5(4v^2+5v_{xx}^2+v_{xx}^2),
 \end{equation}
 and proceeding as for the estimate \eqref{PP2} with the help of \eqref{147}-\eqref{148} and the exponential decay of $\varphi_{c{_j}}$ on $D_j$, it holds
    \begin{eqnarray}
 \|v(t,\cdot)\|_{C^1(D_j)}+ \|u(t,\cdot)\|_{L^\infty(D_j)}& \lesssim &  \sum_{j=-N_-\atop j\neq 0}^{N_+} \|\varphi_{c_j}(\cdot -{ x_j(t))}\|_{L^\infty(D_j)}
   + \|u-\sum_{j=-N_-\atop j\neq 0}^{N_+}\varphi_{c_j}(\cdot -{ x_j(t)})\|_{L^\infty(D_j)}\nonumber\\
  & + &\sum_{j=-N_-\atop j\neq 0}^{N_+} \|\rho_{c_j}(\cdot -{ x_j(t))}\|_{L^\infty(D_j)}
   + \|v-\sum_{j=-N_-\atop j\neq 0}^{N_+}\rho_{c_j}(\cdot -{ x_j(t)})\|_{L^\infty(D_j)} \nonumber \\
    & \le & O(e^{-L/8}) +O(\sqrt{\alpha}) \; .\label{171}
    \end{eqnarray}
    
Now to  estimate $F_1$, we note that combining  \eqref{169}-\eqref{171} and the exponential decay of $\Psi_{j,K}'$ on $D_j^c$, we get
 \begin{align*}
 F_1&\leq 4(\Vert u\Vert_{L^{\infty}(D_j)}+\Vert v\Vert_{L^{\infty}(D_j)}) \int_{\R}u^2\Psi_{j,K}'dx+
  4(\Vert u\Vert_{L^{\infty}(D_j^c)}+\Vert v\Vert_{L^{\infty}(D_j^c)}) \Vert \Psi_{j,K}'\Vert_{L^{\infty}(D^c_j)}\Vert u\Vert_{L^2(\R)}^2\\
 &\leq 20(\Vert u\Vert_{L^{\infty}(D_j)}+\Vert v\Vert_{L^{\infty}(D_j)})\int_{\R}(4v^2+5v_x^2+v_{xx}^2)\Psi_{j,K}'dx+\frac{20}{K}\Vert u_0\Vert_{\mathcal{H}}^3e^{-\frac{1}{6K}(\sigma(\vec{c}) t+L/8)},
 \end{align*}
 where we used   \eqref{u0uinf} and that, thanks to \eqref{EE} and \eqref{1.5}, 
 $$
 \|v\|_{L^\infty(\R)} \le \frac{1}{\sqrt{2}} \| v\|_{H^1} \le \frac{1}{2\sqrt{2}} \|u\|_{\H} =\frac{1}{2\sqrt{2}}\|u_0\|_{\H} \; .
 $$
 Therefore, for $0<\alpha<\alpha_0(\vec{c})\ll1$ small enough and $L>L_0>0$ large enough, it holds
 $$
 F_1\leq \frac{c_{1}}{10}\int_{\R}(4v^2+5v_x^2+v_{xx}^2)\Psi_{j,K}'(x-y_j(t))dx+\frac{C}{K}e^{-\frac{1}{6K}(\sigma(\vec{c})t+L/8)}.\\
$$
  Let us now tackle the estimate of  $F_2$. We first remark that from the definition of $ \Psi $ in Section \ref{Sectmonotonie}, and in particular \eqref{psi3}, we have  for $ K\geq 1 $,
  \begin{equation}\label{174}
(1-\partial_x^2) \Psi_{j,K}' \geq (1-\frac{1}{2 K^2}) \Psi_{j,K}' \Rightarrow
(1-\partial_x^2)^{-1} \Psi_{j,K}'\leq \big(1-\frac{1}{2 K^2}\big)^{-1} \Psi_{j,K}'
  \end{equation}
  and, by Young's convolution estimates and \eqref{1.5},
\begin{equation}\label{174b}
\|h\|_{L^2} \le \frac{1}{2} \|e^{-|\cdot|}\|_{L^2} \|u^2\|_{L^1}\le \frac{1}{2} \|u\|_{L^2}^2 \le 2 \|u\|_{\H}^2 \; .
\end{equation}   We also notice that 
    \begin{eqnarray*}
 h(x)&  = & \frac{1}{2} e^{-x} \int_{-\infty}^x e^{x'}  u^2(x') \, dx' +\frac{1}{2} e^{x} \int_{-\infty}^x e^{-x'}  u^2(x') \, dx',
 \end{eqnarray*}
 and
 $$
 h_x(x)=-\frac{1}{2} e^{-x} \int_{-\infty}^x e^{x'}  u^2(x') \, dx' +\frac{1}{2} e^{x} \int_{-\infty}^x e^{-x'}  u^2(x') \, dx',
 $$
  so  that 
\begin{equation}\label{176}
\vert h_x(x)\vert\leq h(x) \quad\forall x\in\R.
\end{equation}
and thus 
\begin{equation}\label{177}
F_2 \le \int_{\R} (5v+|v_x|) h \Psi_{j,K}' \quad .
\end{equation}
Therefore, according to \eqref{170}-\eqref{171} and \eqref{174}-\eqref{177}, we have
\begin{align*}
F_2&\leq 6 \Vert v\Vert_{C^1(D_j)}\int_{\R}h\Psi_{j,K}'dx+\Vert \Psi_{j,K}'\Vert_{L^{\infty}(D_j^c)}
\|h\|_{L^2} (5\|v\|_{L^2}+\|v_x\|_{L^2})\\
&\leq 6 \Vert v\Vert_{C^1(D_j)}\int_{\R}u^2(1-\partial_x^2)^{-1}\Psi_{j,K}'dx+
2  \Vert \Psi_{j,K}'\Vert_{L^{\infty}(D_j^c)}  \|u\|_{\H}^2 (5\|v\|_{L^2}+\|v_x\|_{L^2})\\
&\leq 30  \Vert v\Vert_{C^1(D_j)}\big(1-\frac{1}{2K^2}\big)^{-1} \int_{\R}(4v^2+5v_x^2+v_{xx}^2)\Psi_{j,K}'dx+6\Vert \Psi_{j,K}'\Vert_{L^{\infty}(D_j^c)}
\|u\|_{\H}^3 \; .
\end{align*}
Using \eqref{1.5} and the exponential decay of of $\Psi_{j,K}'$ on $D_j^c$ with \eqref{169}, we thus get  
$$
F_2\leq 30\big(1-\frac{1}{2K^2}\big)^{-1} \Vert v\Vert_{C^{1}(D_j)}\int_{\R}(4v^2+5v_x^2+v_{xx}^2)\Psi_{j,K}'dx+6\Vert u_0\Vert_{\mathcal{H}}^3e^{-\frac{1}{6K}(\sigma(\vec{c})+L/8)},
$$
so that by \eqref{177} $F_2$ satisfies \eqref{168}  for $0<\alpha<\alpha_0(\vec{c})\ll1$ small enough and $ L>L_0\gg1$ large enough.\\

 To estimate $F_{3}$, remark that using \eqref{166} and \eqref{170} one may write
 \begin{align*}
 F_{3}\leq \frac{15}{2}\lambda c_{N_+}\Vert u\Vert_{L^{\infty}(D_j)}\int_{\R}(4v^2+5v_x^2+v_{xx}^2)\Psi_{j,K}'dx+\frac{3}{2}\lambda c_{N_+}\Vert\Psi_{j,K}'\Vert_{L^{\infty}(D_j^c)}\Vert u\Vert_{L^{\infty}(\R)}\Vert u\Vert_{L^2(\R)}^2 \; .
 \end{align*}
 Using that, by hypothesis $ 0\le \lambda\le (2c_1)^{-1} $, the exponential decay of $\Psi_{j,K}'$ on $D_j^c$, \eqref{1.5} and \eqref{u0uinf},  we deduce that $F_{3}$ satisfies \eqref{168} for $0<\alpha<\alpha_0(\vec{c})\ll1$ small enough and $ L>L_0(\vec{c})\gg1$ large enough. \\

Finally,  $ \Psi_{j,K}'\geq 0$, $ \lambda \geq 0 $ and \eqref{176} ensure that $F_{4}+F_{5}$ is non positive. Gathering \eqref{166}-\eqref{168}
 we thus infer that 
$$
\frac{d}{dt}\mathcal{J}_{j,\lambda,K}(t)\leq \frac{C}{K}\Vert u_0\Vert_{\mathcal{H}}^3e^{-\frac{1}{6K}(\sigma(\vec{c})t+L/8)}.
$$
Integrating this inequality between $ 0 $ and $ t $, \eqref{monotonicityestim} follows and this proves the proposition for smooth initial solutions. Finally, approximating the initial data as in \eqref{app}, the strong continuity result with respect to initial data \eqref{cont2}  in Proposition \ref{weakGWP} ensures that \eqref{monotonicityestim} also hold for 
 $ u_0\in Y $ satisfying Hypothesis \ref{hyp}.
\hfill $\square$ \vspace*{2mm} \\
We will also need the following monotonicity result on $ E+\gamma M $ at the right of the curve $ y_1(\cdot) $.
We introduce the function $ \phi $ defined by 
\begin{equation}\label{defPhi}
\Phi(x) =\left\{ \begin{array}{rcl}
0 &\text{for}& x\le 0\\
x/2  &\text{for}& x\in [0,2] \\
1&\text{for}& x\ge 2
\end{array}
\right.
\end{equation}

\begin{lemma} \label{lemmo} Let $u\in C([0,T]; H^{\infty} \cap L^\infty(0,T; Y) $ be the solution to \eqref{DP} satisfying Hypothesis \ref{hyp} and \eqref{hyplem} for some $ L>0 $. Assume moreover that $u$ satisfies
\begin{equation}\label{ttt}
x_0(t) \le x_1(0)-L/4 +\frac{c_1 t}{2} \; , \quad \forall t\in [0,T],
\end{equation}
where $ x_1(\cdot) $ is defined in Lemma \ref{eloignement}. There exists $ L_0=L_0(\vec{c})  >0 $ such that if $ L\ge L_0  $ then on $ [0,T] $, it holds 
\begin{equation}\label{154}
\int_{\R} \Bigl( 4v^2+5v_x^2+v_{xx}^2\Bigr)(t) \Psi(\cdot -y_1(t)) +\frac{c_1}{2^9} y \Phi(\cdot -y_1(t)) \le \int_{\R} \Bigl( 4v_0^2+5v_{0,x}^2+v_{0,xx}^2\Bigr) \Psi(\cdot -y_1(0)) + \frac{c_1}{2^9}y(0)\Phi(\cdot -y_1(0))+ O(e^{-\frac{L}{48}})
\end{equation}
where $ y_1(\cdot) $ is defined in \eqref{defy1} and $ \Psi $ is defined in \eqref{defPsi}.
\end{lemma}
\begin{proof}
 Applying \eqref{163} with $ g(t,x)=\Psi(x - y_1(t)) $ and \eqref{go2} with  $ g(t,x)=\Phi(x - y_1(t)) $   and recalling the definition \eqref{defy1} of $ y_1(\cdot) $, we get
  \begin{equation}
 \frac{d}{dt}\big[E(u)+\frac{c_1}{2^9}M(u)\big] = -\frac{c_1}{2} \int_{\R}  \Bigl[ \Psi'(4v^2+5v_x^2+v_{xx}^2) +\frac{c_1}{2^9} \Phi'  y \Bigr] +\frac{3 c_1}{2^{10}} \int_{\R} (u^2-u_x^2) \phi' +\frac{c_1}{2^9}\int_{\R}uy\Phi'+J
 \end{equation}
 where  thanks to \eqref{168}, 
 \begin{eqnarray}
 J & = &\int_{\R}(\frac{2}{3}u-4v)u^2\Psi'dx
+\int_{\R}(5vh+v_x h_x)\Psi 'dx \nonumber \\
& \le &  \frac{c_1}{2^4} \int_{\R}   (4v^2+5v_x^2+v_{xx}^2)\Psi' + C \, \Vert u_0\Vert_{\mathcal{H}}^3e^{-\frac{1}{6}(\sigma(\vec{c})+L/8)} \; .
 \end{eqnarray}
  We first observe that 
\begin{equation}\label{za}
  \int_{\R} (u^2-u_x^2) \Phi' \le    \int_{\R} u^2  \Phi' = \int_{\R} (4v-v_{xx})^2  \Phi' 
  \le 5  \int_{\R} \Bigl( 4v^2+5v_x^2+v_{xx}^2\Bigr) \Phi'\; ,
  \end{equation}
  where, according to the definition \eqref{defPhi} of $ \Phi $, it holds 
  \begin{equation}\label{za0}
\frac{3c_1}{2^{10}} \, 5  \Phi' \le \frac{c_1}{4}  \Psi' \quad\text{ on } \R \; .
  \end{equation}
Second,  \eqref{ttt} together with  \eqref{defy1} and the definition \eqref{defPhi} of $ \Phi $ ensure that $ y(t,\cdot) $ is non negative on the support of $
 \Phi'(\cdot-y_1(t))$ that is $[y_1(t),y_1(t)+2] $. Therefore \eqref{171} leads to 
 \begin{eqnarray}\label{za1}
 \frac{c_1}{2^9}\int_{\R}uy\Phi' & \le & \frac{c_1}{2^9}\Vert u\Vert_{L^{\infty}(]y_1(t),y_1(t)+2[)}\int_{\R}y\Phi' dx
 \le \frac{c_1}{2^{11}} \int_{\R}y\Phi' \, dx \; .
 \end{eqnarray}
Therefore \eqref{za}-\eqref{za1} and\eqref{147} we obtain
 $$
 -\dot{y}_1(t) \int_{\R}  \Bigl[ \Psi'( 4v^{2}+5v^{2}_{x}+v^{2}_{xx}) +\frac{c_1}{2^9} \Phi' y \Bigr]  +\frac{3c_1}{2^{10}} \int_{\R} (u^2-u_x^2) \Phi'  +\frac{c_1}{2^9}\int_{\R}uy\Phi' 
 \le -\frac{c_1}{4} \int_{\R} \Psi'( 4v^2+5v_x^2+v_{xx}^2)dx  \; 
 $$
 that leads to 
 $$
 \frac{d}{dt}\big[E(u)+\frac{c_1}{2^9}M(u)\big] \le  C \, \Vert u_0\Vert_{\mathcal{H}}^3e^{-\frac{1}{6}(\sigma(\vec{c})+L/8)} \; .
 $$
 This proves \eqref{154} by integrating in time.
  \end{proof}
\subsection{Control of the growth of $ \|y\|_{L^1} $}
The control of the growth of the mass of $ y(t) $ is more delicate than in the case of the stability of a single peakon. Indeed, in this last case we deeply use that $ u $ stays $ L^\infty $-close to the peakon that is positive and thus the negative part of $ u$ stays small. In the present case, this is of course no more true because our train of antipeakon-peakons is no more positive. To overcome this difficulty we make use of the  monotony argument for $ E(u) +\gamma M(u) $ proven in Lemma \ref{lemmo}.
\begin{proposition}\label{prop5}
Let $u_0\in Y\cap H^\infty(\R) $ satisfying Hypothesis 1 and  $ u\in C(\R_+;H^\infty) \cap L^\infty_{loc}(\R_+; Y) $ be the associated solution to DP given by Proposition \ref{weakGWP}. There exist $ \alpha_0=\alpha_0(\vec{c}) $ and $ L_0=L_0(\vec{c}) $ such that
 if 
\begin{equation}\label{assum11}
u(t)\in U(\alpha, L,\vec{c}) \quad , \quad \forall t\in [0,T]
\end{equation}
with $ 0<\alpha\le \alpha_0 $ and $ L\ge L_0 $ then
\begin{equation}\label{prop4eq1}
\Vert y(t,\cdot)\Vert _{L^1(\R)}     \leq    e^{17+2^{-5} (c_1\wedge |c_{-1}|)  t}
 \frac{(\|\vec{c}\|_1 +1)^3}{(c_1 \wedge |c_{-1}|)^2} (1+\|y_0\|_{L^1}) \; \quad , \quad \forall t\in [0,T].
\end{equation}
\end{proposition}
\begin{proof}
In view of Lemma \ref{eloignement}, there exists $ N_-+N_+ $  $C^1$-functions $ x_{-N_-}(\cdot)< ..<x_{-1}(\cdot)< x_1(\cdot)< ..<x_{N_+}(\cdot) $ defined on $[0,T] $  that satisfy  \eqref{149}-\eqref{eloi} and  \eqref{147}-\eqref{148}.

We separate two cases depending on the place of $ x_0(0)$ with respect to $x_{1}(0)$. \\
{\bf Case 1.} $x_0(0)\le x_{1}(0)-L/3 $.
Then according to \eqref{149}, \eqref{147},  the definition \eqref{defq} of $ x_0(\cdot) $ and a continuity argument, $x_0(t)\le x_1(t) -L/3 $ and in particular 
 $ \dot{x}_0(t)\le c_1/2 $ for all $ t\in [0,T]$.  This ensures that 
$$ 
 x_0(t)+L/12 \le y_1(t)= x_1(0)-L/4 +\frac{c_1 }{2}t , \quad \forall t\in [0,T],
$$
where $ y_1(\cdot)$ is defined in \eqref{defy1}.\\
Therefore Lemma \ref{lemmo} leads to 
$$
\int_{\R} (4v^2+5v_x^2+v_{xx}^2) \Psi(\cdot-y_1(t))+\frac{c_1}{2^9}\int_{\R} y\phi(\cdot-y_1(t)) \le \int_{\R} (4v^2+5v_x^2+v_{xx}^2)\Psi(\cdot-y_1(0))+\frac{c_1}{2^9}\int_{\R} y(0)\phi(\cdot-y_1(0))+O(e^{-\frac{L_0}{48}})\;
$$
 Making use of the conservation of $ E $ and of the definition of $ \Psi $, if follows that for $ L $ large enough, 
$$
\int_{y_1(t)+2}^{+\infty} y(t,x) \, dx \le \frac{2^9}{c_1} E(u_0) + \|y_0\|_{L^1} +O(e^{-\frac{L_0}{48}}) \le 1+\frac{2^9}{c_1} E(u_0) + \|y_0\|_{L^1} \; , \quad \forall  t\in [0,T] \; .
$$
On the other hand, according to \eqref{dodo}, $ u_x\ge -u $ on $ ] x_0(t),+\infty[$ and by Lemma \ref{LemmaPP1loc} $ \forall t\in [0,T]$ we have,
$$
u(t) \le \sum_{i=1}^{N_+}c_i +O( \sqrt{\alpha_0} ) \quad \text{on} \quad [y_1(t)+2,+\infty] \quad  \text{and} \quad
u (t)\le   O ( \sqrt{\alpha_0} ) +O(e^{-L_0/8}) \le O(\sqrt{\alpha_0})  \quad \text{on} \quad [x_0(t),y_1(t)+2]  \; ,
$$
where to get the last inequality we take $ L_0 >0 $ such that $ O(e^{-L_0/8}) \le \sqrt{\alpha_0} $.
Therefore, according to  \eqref{DP2} and \eqref{13c}, we have
\begin{align*}
\frac{d}{dt}\int_{\R}y^{+}(t,x)dx&=\frac{d}{dt}\int_{q(t,x_0)}^{+\infty}y(t,x)dx=- 2\int^{+\infty}_{x_0(t)}u_x(t,x)y(t,x)dx\\
&\leq2\int^{y_1(t)+2}_{x_0(t)} u(t,x) y(t,x)dx+2\int_{y_1(t)+2}^{+\infty}u(t,x)y(t,x)dx\\
&\leq  2\big(\|\vec{c}\|_{1}+O(\sqrt{\alpha_0})\big) \big(1+\frac{2^9}{c_1} E(u_0)+\|y_0\|_{L^1}\big)+ O(\sqrt{\alpha_0})  \int_{\R}y^{+}(t,x)dx.
\end{align*}
Hence, Gr$\ddot{\text{o}}$nwall's inequality yields $ \forall t\in [0,T]$
\begin{equation}
\int_{\R}y^{+}(t,x)dx\leq   e^{C \sqrt{\alpha_0}   t}\Bigl( \Vert y_0 \Vert_{L^1} +  2t \, \big(\|\vec{c}\|_{1}+1\big) \big(1+
\frac{2^9}{c_1} E(u_0)+\|y_0\|_{L^1}\big)\Bigr) ,
\end{equation}
for some universal constant $ C>0 $. Since, according to Proposition \ref{smoothWP},  $\dsp M(u)=\int_{\R} y $ is conserved for positive times,
 it follows that 
 \begin{equation}\label{l1yy}
\Vert y(t,\cdot)\Vert _{L^1(\R)}  \leq   2 e^{C \sqrt{\alpha_0}   t}\Bigl( \Vert y_0 \Vert_{L^1} +  2t \, \big(\|\vec{c}\|_{1}+1\big) \big(1+\frac{2^9}{c_1} E(u_0)+\|y_0\|_{L^1}\big)\Bigr) .
\end{equation}
Taking $ \alpha_0\le (c_1\wedge |c_{-1}|)^2 (C\,  2^{10})^{-2} $ we thus deduce that 
$$
\Vert y(t,\cdot)\Vert _{L^1(\R)}     \leq   2 e^{2^{-10} (c_1\wedge |c_{-1}|) t}\Bigl( \Vert y_0 \Vert_{L^1} +
2 t \big(\|\vec{c}\|_{1}+1\big) \big(1+\frac{2^9}{c_1} E(u_0)+\|y_0\|_{L^1}\big)\Bigr) \; .
$$
Since for $ t\ge 0$, $ t  e^{2^{-10} (c_1\wedge |c_{-1}|) t}\le t e^{2^{-6}(c_1\wedge |c_{-1}|) t} e^{2^{-5} (c_1\wedge |c_{-1}|) t}\le \frac{e^{2^{-5} (c_1\wedge |c_{-1}|) t}}{2^{-6}(c_1\wedge |c_{-1}|)} e^{-1}$, it follows that  
 \begin{equation}\label{l1y2}
 \Vert y(t,\cdot)\Vert _{L^1(\R)} \le 2 e^{2^{-5} (c_1\wedge |c_{-1}|)  t}\Bigl( \Vert y_0 \Vert_{L^1} +
2^{7} \frac{\big(\|\vec{c}\|_{1}+1\big)}{(c_1\wedge |c_{-1}|) } \big(1+\frac{2^9}{(c_1\wedge |c_{-1}|) } E(u_0)+\|y_0\|_{L^1}\big)\Bigr) \; .
\end{equation}
Finally, taking $ \alpha_0\le 1 $, \eqref{assum11} ensures that $ E(u_0) \le (\|\vec{c}\|_1+1)^2 $, and noticing that 
$$
\frac{\|\vec{c}\|_1+1}{c_1\wedge |c_{-1}|} \ge 1 , 
$$
we eventually get  \eqref{prop4eq1}.

{\bf Case 2: } $x_0(0)\ge x_{1}(0)-L/3 $. Then by \eqref{eloi}, we must have  $ x_0(0)\ge x_{-1}(0)+L/3 $.

In this case, we make use of the fact  that the  DP equation is invariant by the change of unknown $ u(t,x)\mapsto \tilde{u}(t,x)=-u(t,-x) $. Clearly $ \tilde{u}(0) $ also satisfies hypothesis \ref{hyp} with $ \tilde{x}_0(t)=-x_0(t) $.
 Morever, $\tilde{u} $ satisfies \eqref{147} on $ [0,T] $ with $ N_-$ and $ N_+ $ respectively replaced by
  $ \tilde{N}_-=N_+ $ and $ \tilde{N}_+=N_- $, $ c_i $ replaced by $\tilde{c}_i=-c_{-i}$  and  $ x_i(t) $ replaced by $ \tilde{x}_i(t)=-x_{-i}(t)$. In particular, it holds 
   $$\tilde{x}_0(0)=-x_0(0)\le -x_{-1}-L/3=\tilde{x}_1(0) -L/3 ,
   $$
   and thus $\tilde{u} $ satisfies the hypothesis of Case 1. Therefore $ \tilde{y}=\tilde{u}-\tilde{u}_{xx} 
   =-y(t,-\cdot)$ satisfies \eqref{l1y2} with $ c_{-1} $ and $ c_1 $ respectively replaced by 
  $ \tilde{c}_{-1} =-c_1$  and $ \tilde{c}_1=-c_{-1} $. This completes the proof of  \eqref{prop4eq1} .
\end{proof}

Let us now state the adaptation of Proposition \ref{pro2} in the present case. The role of $ x(\cdot) $ will be now play by 
 $ x_1(\cdot) $ that localizes the slowest peakon. The proof is essentially the same as the one of Proposition \ref{pro2}.
  However, in the present case \eqref{50b} is not available anymore on $ \R $ but we actually only need that 
  it holds on $ [x_{-1}(t)+L/4, +\infty[ $ that is verified since $\sum_{j=-N_-\atop j\neq 0}^{N_+}\varphi_{c_j}  \ge
 O(\sqrt{\alpha_0})+O(e^{-L_0/8})
  $ on this interval. 

\begin{proposition}\label{pro2loc}
There exists $\alpha_0(\vec{c})>0 $  and $ L_0(\vec{c})>0 $ such that for any  $u_0\in Y\cap H^\infty(\R) $ satisfying Hypothesis \ref{hyp}, if  the solution $u\in C(\R_{+}; H^\infty(\mathbb{R})\big)$ emanating from $u_0$ satisfies for some
  $0<\alpha<\alpha_0$, $L\ge L_0 $ and  $T >0 $,
\begin{equation}\label{assum4loc}
u\in U\Big(\alpha,L/2,\vec{c}\Big) \quad \text{ on } [0,T], 
\end{equation}
  then
 for all $t\in[0,T]$,
\begin{equation}\label{81loc}
\Vert y^{-}(t,\cdot)\Vert_{L^1(]x_1(t)-\frac{1}{16}c_{1}t,+\infty[)}\leq e^{-c_{1}t/8}\Vert y_0\Vert_{L^1(\R)} ,
\end{equation}
where $y^- =\max(-y,0)$, and $x_1(\cdot)$ is the $C^1$-function 
constructed in Lemma \ref{eloignement}.
Moreover it holds 
\begin{equation}\label{lolo}
u(t,\cdot)-6v(t,\cdot) \le   e^{27-\frac{c_1t}{32}}  
 \frac{(\|\vec{c}\|_1 +1)^3}{(c_1 \wedge |c_{-1}|)^2} (1+\|y_0\|_{L^1})  \quad \text{on} \; ]x_1(t)-8, +\infty[ \;, 
\end{equation}
where $v=(4-\partial_x)^{-1} u $ .

\end{proposition}
\begin{proof} As mentioned above we mainly proceed as in Proposition  \ref{pro2} but with $ x(\cdot) $ replaced by $ x_1(\cdot) $. 
Hence, for $ t\in [0,T]$, we separate two possible cases  according to the distance between $x_0(t/2) $  and $ x_1(t/2) $.\\
\noindent
{\it Case 1:}
\begin{equation}\label{79b}
x_0(t/2)<x_1(t/2)-\ln(3/2).
\end{equation}
In this case, the same continuity argument as in the proof of Proposition \ref{pro2} ensures that 
\begin{equation}
x_1(t)-x_0(t)\geq\ln(3/2)+\frac{c_1}{16}t \; .
\end{equation}
This proves that $y^{-}(t,\cdot)=0$ on $]x_1(t)-\frac{1}{16}c_1t,+\infty[$ and thus that  \eqref{81loc} holds in this case.\\
\noindent
{\it Case 2:}
\begin{equation}\label{83b}
x_0(t/2)\geq x_1(t/2)-\ln(3/2).
\end{equation}
Then, as in  the proof of Proposition \ref{pro2}, \eqref{81loc} is a consequence of the two following estimates :
\begin{equation}\label{91b}
\left\vert \int_{x_0(t/2)-\ln 2}^{x_0(t/2)}y(t/2,s)ds\right\vert\leq e^{-\frac{1}{4}c_1 t}\Vert y_0\Vert_{L^1(\R)}
\end{equation}
and
\begin{equation}\label{94b}
\left\vert \int_{x_1(t)-\ln (3/2)-\frac{c_1}{16} t }^{x_0(t)}y(t,s)\, ds\right\vert \le 
e^{c_1 t/8} \left\vert \int_{x_0(t/2)-\ln 2}^{x_0(t/2)}y(t/2,s)\, ds\right\vert \; .
\end{equation}
\eqref{91b} can be obtained exactly as \eqref{91} in Proposition \ref{pro2}. We thus focus on \eqref{94b} where there is the main change. Indeed, we are not allowed to use \eqref{50b} in order to prove  the crucial estimate \eqref{13bbb}.
 The idea to overcome this difficulty  is to notice that actually we only need such estimate from below on $ u $  in  $[x_{-1}(t)+L/4, +\infty[ $. 

Indeed, let $ q_t$ be the flow-map defined in \eqref{defq1}. For $L$ large enough, $ \eqref{eloi} $ and \eqref{83b} ensure  that $ x=q_{t/2}(t/2,x)\ge x_{-1}(t/2)+L/2 $ as soon as $ x\in [x_0(t/2)-\ln 2, x_0(t/2)]$. Therefore, by \eqref{149}, \eqref{148},  \eqref{PP3loc} and a continuity argument, 
  for $ \tau \in [t/2,T] $ it holds 
$$
q_{t/2}\bigl(\tau, x\bigr)-x_{-1}(\tau)\ge L/2,  \qquad  \forall x\in[x_0(t/2)-\ln 2,x_0(t/2)] \; .
$$
On the other hand,  \eqref{148} and \eqref{PP3loc} ensure that for all $\tau\in [0,T] $, 
$$
 u(\tau,x)  \ge - 2^{-5} c_1 \quad \text{on} \quad [x_{-1}(\tau) +L/4,+\infty[ \; .
$$
Combining the two above estimates with \eqref{dodo} we obtain as in Proposition \ref{pro2} that  for any $ \tau \in [t/2, t] $ and any 
 $x\in[x_0(t/2)-\ln 2,x_0(t/2)] $, 
\begin{equation}\label{13bbbb}
\partial_x q_{t/2}(t,x)\ge exp\Big(-\int_{t/2}^t  2^{-5} c_1 \,  ds\Big)\ge e^{-2^{-4} c_1 t }\; .
\end{equation}
Once we have the above estimate, the rest of the proof of \eqref{81loc} follows the same lines as in the proof of Proposition \ref{pro2}. 

Finally to prove \eqref{lolo},  we take $ \alpha_0 $ and $ L_0 $  that are suitable for Proposition \ref{prop5} . \eqref{ql} together with \eqref{81loc} , \eqref{prop4eq1}  ensure that for $ x\ge x_1(t)-8 $ it holds
\begin{eqnarray*}
6v(x)-u(x) & \ge &  - \frac{1}{2}\int_{-\infty}^{x_1(t)-\frac{c_1}{16}t} e^{-|x-z|} y^-(z)\, dz- \frac{1}{2}
\int_{x_1(t)-\frac{c_1}{16}t}^{+\infty} e^{-|x-z|} y^-(z)\, dz\\
& \ge & - e^{0\wedge (8-\frac{c_1}{16}t)}     e^{2^{-5} (c_1\wedge |c_{-1}|) t} \, e^{17} \,   \frac{(\|\vec{c}\|_1 +1)^3}{(c_1 \wedge |c_{-1}|)^2} (1+\|y_0\|_{L^1})  \\
& & -\frac{1}{2} e^{-c_1t/8}\Vert y_0\Vert_{L^1(\R)}
\\
 & \ge & - e^{18}\, e^{9-\frac{c_1t}{32}} \, 
 \frac{(\|\vec{c}\|_1 +1)^3}{(c_1 \wedge |c_{-1}|)^2} (1+\|y_0\|_{L^1}) \; .
 \end{eqnarray*}
 \end{proof}
 \subsection{An approximate solution}
 A new difficulty with respect to the case of a single peakon will be that 
 $$
  t\mapsto \sum_{j=-N_-\atop j\neq 0}^{N_+}  \varphi_{c_j}(\cdot-z_j^0-c_{j} t ) 
  $$
   is not an exact solution of the DP equation. The aim of the  following lemma is to overcome this difficulty by proving that if $ L>0 $ is large enough then this is an approximate  solution with an error in $ L^2(\R) $ of order $ e^{-L/2} $  on a time interval of  order $ \ln(L^{3/4})$.
 \begin{lemma}\label{exsol}
Let be given $ N_- \in \N^*$ negative   velocities $ c_{-N_-} <..<c_{-1}<0 $, $ N_+ \in \N^*$ positive velocities $0<c_1<..<c_{N_+} $ and  $z_{-N_-}^0<..z_{-1}^0 <z_1^0<..<z_{N_+}^0 $.
There exists  $L_0>0 $ only depending on $ \vec{c} $ such that for any $ L\ge L_0 $ if 
\begin{equation} \label{zz}
z_i^0-z_j^0\ge L \quad \text{for} \quad i> j 
\end{equation}
then the solution $ u $ to \eqref{DP} emanating from 
$\dsp u_0=\sum\limits_{\substack{\scriptscriptstyle j=-N_- \\ \scriptscriptstyle j \neq 0}}^{\scriptscriptstyle N_+} \varphi_{c_j}(\cdot-z_j^0) $ satisfies 
$$
\sup_{t\in [0, 2^5 (c_1\wedge |c_{-1}|)^{-1} \ln( L^{3/4})]} \Bigl\|u(t)-\sum_{j=-N_-\atop j\neq 0}^{N_+}  \varphi_{c_j}(\cdot-z_j^0-c_{j} t ) 
\Bigr\|_{\H} \le e^{-L/2}\; .
$$
\end{lemma}
\begin{proof}
We set $\dsp \underline{u}(t)=\sum\limits_{\substack{\scriptscriptstyle j=-N_- \\ \scriptscriptstyle j \neq 0}}^{\scriptscriptstyle N_+} \varphi_{c_j}(\cdot-z_j^0-c_{j} t ) $. Using that  $ \varphi_{c}(x-ct) $ is a solution to \eqref{DP}, one can check that $ \underline{u}$ satisfies 
\begin{equation}\label{toy}
\underline{u}_t+ \underline{u} \underline{u}_x + \frac{3}{2} \partial_x (1-\partial_x^2)^{-1} (\underline{u}^2) =F 
\end{equation}
with 
$$
F:= \sum_{i<j} c_i c_j  \Bigl( 1 + 3 (1-\partial_x^2)^{-1} \Bigr) \partial_x \Bigl(\varphi(\cdot- z_i^0-c_{i} t)
\varphi(\cdot - z_j^0-c_{j} t)\Bigr)\; .
$$
On account of \eqref{zz}, straightforward calculations lead to 
$$
\sup_{t\in [0,T]} \big\|\partial_x \Bigl(  \varphi(\cdot- z_i^0-c_{i} t)
\varphi(\cdot- z_j^0-c_{j} t)\Bigr) \|_{L^1}+\|\partial_x^2 \Bigl(  \varphi(\cdot- z_i^0-c_{i} t)
\varphi(\cdot- z_j^0-c_{j} t)\Bigr) \big\|_{\M}
 \lesssim (L+1) e^{-2L/3}\; . 
$$
so that 
\begin{equation}\label{toy2}
\sup_{t\in \R_+} \|F(t)\|_{L^1}+\|F_x(t)\|_{\M}\lesssim (L+1) e^{-2L/3} \; . 
\end{equation}
Note also that for all $ t\ge 0 $ it holds $\sup_{t\in [0,T]}\|(\underline{u}-\underline{u}_{xx})(t)\|_{\M} = \|\vec{c}\|_{1}$. 

Now, since $ \underline{u}(0)=u_0 $ clearly satisfies Hypothesis \ref{hyp}, the solution $ u$ to \eqref{DP} emanating from $ 
u_0=\underline{u}(0) $ exists for all positive times in $ Y $. For $ T>0 $ we set 
 $$
 M_T=\sup_{t\in [0,T]}\|u-u_{xx}\|_{\M} \; .
 $$
At this stage it s worth noticing that Proposition \ref{prop5} ensures that 
 \begin{equation} \label{MT}
  M_T \le e^{17+2^{-5} (c_1\wedge |c_{-1}|)  T}
 \frac{(\|\vec{c}\|_1 +1)^4}{(c_1 \wedge |c_{-1}|)^2} \; .
 \end{equation}
Setting $ w=u-\underline{u}$, using exterior regularization and  proceeding as in \cite{CM} (see also \cite{MR2271927} for the DP equation pp$\colon$480-482), 
 we get on $ [0,T] $
 \begin{align*}
  \frac{d}{dt} \Bigl( \int_{\R} |\rho_n \ast w |+   |\rho_n \ast w_x | \Bigr) 
  \lesssim & (M_T +\|\vec{c}\|_{1} ) \Bigl( \int_{\R} |\rho_n \ast w |+ |\rho_n \ast w_x | \Bigr) \\
  & + \int_{\R} (|\rho_n \ast F |+  |\rho_n \ast F_x | )+R_n(t)
  \end{align*}
  where $ (\rho_n)_{n\ge 0} $ is defined in \eqref{rho}, 
  $$
  R_n(t) \to 0 \text{ as } n\to +\infty  \; \text{and } |R_n(t)| \lesssim 1, \; n\ge 1, t\in\R_+.
  $$
Therefore Gronwall inequality and since $w(0)=w_x(0)=0$, yields to 
\begin{equation}
\int_{\R} |\rho_n \ast w(t)|+|\rho_n \ast w_x(t) | \lesssim \int_0^t e^{C (M_T +\|\vec{c}\|_{1} )(t-s)}
\Bigl( \int_{\R} |\rho_n \ast F(s) |+ |\rho_n \ast F_x(s) | +|R_n(s)| \Bigr)ds \; .
\end{equation}
Letting $ n $ tends to $+\infty $ and making use of \eqref{toy2} and then \eqref{MT}, we thus get that for $ L$ large enough 
\begin{eqnarray}\label{89}
\sup_{t\in [0,T]}\|w(t)\|_{\H} \le\sup_{t\in [0,T]}\|w(t)\|_{L^2} \le \sup_{t\in [0,T]}\|w(t)\|_{W^{1,1}} \le e^{C (1+M_T +\|\vec{c}\|_{1} )T} e^{-5 L/8} \; .
\end{eqnarray}
This estimate together with  \eqref{MT} ensure that there exists $ L_0(\vec{c}) \ge 1 $ such that  for all $ L> L_0$, \begin{equation}\label{90}
 \Bigl\|u(t)- \underline{u}(t)
\Bigr\|_{\H} \le e^{-L/2}
\end{equation}
as soon as 
 \begin{equation}\label{9191}
 0\le t \le 2^5 (c_1\wedge |c_{-1}|)^{-1} \ln (L^{3/4}) \; .
\end{equation}
Indeed, as soon as \eqref{90}-\eqref{9191} are satisfied, \eqref{MT} gives
$$
M_T \le e^{17}
 \frac{(\|\vec{c}\|_1 +1)^4}{(c_1 \wedge |c_{-1}|)^2}   L^{3/4} 
$$
so that \eqref{89} leads to 
$$
\|u(t)-\underline{u}(t)\|_{\H} \le
\exp\Bigl( C_1+C_2  \, L^{3/4} \ln( L^{3/4})\Bigr) e^{-5L/8}\; .
$$
where $ C_1=C_1(\vec{c}) >0 $ and $ C_2=C_2(\vec{c})>0 $. This gives  \eqref{90} for $ L $ large enough 
and proves the result by a continuity argument.
\end{proof}

\subsection{ Two global  estimates}\label{global}
The following generalization of the quadratic identity in Lemma \ref{Lemma P2} was proved in \cite{AK}.
\begin{lemma}(Global quadratic identity)\label{lemma20}
Let   $u\in L^2(\R)$ and assume that $ z_{-N_-}<..<z_{-1}<z_1<..<z_{N_+}  $ with 
  $z_i-z_{j}\geq  L/2 $ for $i>j $. Then  it holds
\begin{equation}
E(u)-\sum_{i=-N_{-}}^{N_+}E(\varphi_{c_i})=\left\Vert u-\sum_{i=-N_-}^{N_+}\varphi_{c_i}(\cdot-z_i)\right\Vert _{\mathcal{H}}^2+4\sum_{i=-N_{-}}^{N_+}c_i\left(v(z_i)-\frac{c_i}{6}\right)+O(e^{-L/2})\quad i\in [[-N_{-},N_+]]\setminus \{0\} .
\label{203}
\end{equation}
where $ v=(4-\partial_x^2)^{-1} u $.
\end{lemma}

\begin{proof}
First, according to the definition of the energy space \eqref{EE} we notice that
\begin{align}
\left\Vert u-\sum_{i=-N_-}^{N_+}\varphi_{c_i}(\cdot-z_i)\right\Vert _{\mathcal{H}}^2&=E(u)+E\left(\sum_{i=-N_-}^{N_+}\varphi_{c_i}(\cdot-z_i)\right)-2\sum_{i=-N_-}^{N_+}\big \langle(1-\partial_x^2)\varphi_{c_i}(\cdot -z_i), v
\big\rangle \nonumber\\
&=E(u)+E\left(\sum_{i=-N_-}^{N_+}\varphi_{c_i}(\cdot-z_i)\right)-4\sum_{i=-N_-}^{N_+}c_iv(z_i).
\label{204}
\end{align}
where  we used that $(1-\partial_x^2)\varphi_{c_i}(\cdot-z_i)=2c_i\delta_{z_i}$ with $\delta_{z_i}$ the Dirac mass applied at point $z_i$. However,
\begin{align}
\displaystyle E\left(\sum_{i=-N_-}^{N_+}\varphi_{c_i}(\cdot-z_i)\right)&=\sum_{i=1}^N\sum_{j=1}^N\langle (1-\partial_x^2)\varphi_{c_i}(\cdot-z_i), \rho_{c_j}(\cdot-z_j)\rangle \nonumber\\
&=2\sum_{i=-N_-}^{N_+}\sum_{j=-N_-}^{N_+}c_i\rho_{c_j}(z_i-z_j)\nonumber\\
&\displaystyle =\frac{1}{3}\sum_{i=-N_-}^{N_+}c_i^2+2\sum_{i=-N_-}^{N_+}c_i\sum\limits_{\substack{j=-N_- \\ j \neq i}}^{N_+}\rho_{c_j}(z_i-z_j) \; .
\label{205}
\end{align}
 From the  definition of $\rho_{c_j}$ in \eqref{1.88} 
  and the fact that $z_i-z_{j}\ge 2L/3 $ for $ i> j $, it follows that
\begin{align}\label{206}
\Bigl| \sum\limits_{\substack{j=-N_- \\ j \neq i}}^{N_+}\rho_{c_j}(z_i-z_j)\Bigr| &=
\Bigl| \sum\limits_{\substack{j=-N_- \\ j \neq i}}^{N_+}\frac{1}{4}\Bigl( e^{-2\vert\cdot\vert}\ast\varphi_{c_j}(\cdot-z_j)\Bigr)(z_i)\Bigr| =\Bigl|\sum\limits_{\substack{j=-N_- \\ j \neq i}}^{N_+}\frac{c_j}{3}e^{-\vert z_i-z_j\vert}-\frac{c_j}{6}e^{-\vert z_i-z_j\vert}\Bigr| \nonumber\\
&\leq\|\vec{c}\|_{1}e^{-2L/3}\leq O(e^{-L/2})
\end{align}
Gathering \eqref{204}, \eqref{205}, \eqref{206} with $E(\varphi_{c_i})=c_i^2/3$ then \eqref{203} holds for $L>L_0\gg1$ large enough.
\end{proof}


The following lemma is an adaptation of  Lemma \ref{LemmaPP} in the present case.

\begin{lemma}\label{LemmaPPloc}
Let $ u\in L^\infty(\R) \cap L^2(\R) $ such that 
\begin{equation}\label{hygloc}
\Big\|u-\sum_{j=-N_-\atop j\neq 0}^{N_+}\varphi_{c_i}(\cdot-z_j)\Big\|_{L^{\infty}(\mathbb{R})}\leq  \frac{10^{-5}}{N_-+N_+} 
(c_1\wedge |c_{-1}|)
\end{equation}
for some $ c_{-N_-}<..<c_{-1}<0<c_1<..<c_{N_+} $ and some $ Z\in \R^{N_-+N_+} $  with 
  $z_i-z_j\geq 2 L/3$ for all $i>j $. Then there exists  $L_0>0$ only depending on  $ \vec{c}$, such that for $ L>L_0\gg1$ large enough, the function $ v=(4-\partial_x^2)^{-1} u$ has got a unique point of local maximum (resp. minimum) $ \xi_i $ on $\Theta_{z_i}
=[z_i-6.7,z_i+6.7]  $
 for any $ 1\le i \le N_+ $(resp. $-N_-\le i \le -1 $).
Moreover, 
\begin{equation}\label{hyg2loc}
\left\Vert u-\sum_{j=-N_-\atop j\neq 0}^{N_+}\varphi_{c_j}(\cdot-\xi_j)\right\Vert_{\mathcal{H}} \le \left\Vert u-\sum_{j=-N_-\atop j\neq 0}^{N_+}\varphi_{c_j}(\cdot-z_j)\right\Vert_{\mathcal{H}} + O(e^{-L/4})\; .
\end{equation}
and 
\begin{equation}\label{hyg33loc}
\xi_i\in{\mathcal V}_{i}=[z_i-\ln \sqrt{2}, z_i+\ln \sqrt{2}] , \quad \forall i\in [[-N_-,N_+]]\setminus \{0\} \; .
\end{equation}
Finally, for any $(y_1,..,y_{N_+})\in \R^{N_+} $, such that 
$$
z_{-1}+L/4 <y_1<z_1<y_2<z_2<\cdot\cdot<y_{N_+} <z_{N_+}
$$
with  $ |y_i-z_j| \ge L/4 $ for $(i,j)\in [[1,N_+]]^2 $ it holds
\begin{equation}\label{hyg3loc}
 \sup_{x\in  ]y_i,y_{i+1}[ \setminus\Theta_{z_i} }\big(|u(x)|, |v(x)|,|v_x(x)|\big) \le \frac{c_1\wedge |c_{-1}|}{100} \; , \quad 
  i\in [[1,N_+]] \; ,
\end{equation}
where we set $ y_{N_+ +1}=+\infty $.
\end{lemma}
\begin{proof}
 Since $z_i-z_j\geq 2 L/3$ for all $i>j $ it holds 
\begin{equation}\label{187}
\sum_{j=-N_-\atop j\neq 0}^{N_+}\rho_{c_j}(x-z_j)=\rho_{c_i}(x-z_i)+
\| \vec{c}\|_{1}O(e^{-L/4}),\quad\forall x\in [z_i-L/3,z_i+L/3] \; .
\end{equation}
Therefore repeating the proof of  Lemma \ref{LemmaPP} on each $ [z_i-L/3,z_i+L/3] $, 
 we obtain that, for $L $ large enough,  the function $ v=(4-\partial_x^2)^{-1} u$ has got a unique point of maximum (resp. minimum) $ \xi_i $ on $\Theta_{z_i}
=[z_i-6.7,z_i+6.7]  $
 for any $ 1\le i \le N_+ $(resp. $-N_-\le i \le -1 $) and that moreover $ \xi_i\in {\mathcal V}_{i}$. In particular, 
 $\xi_i-\xi_j\ge L/2 $  for $ i>j $ and thus applying \eqref{203} for  the $z_i's $ and then the $ \xi_i's $, \eqref{hyg2loc} follows.

\end{proof}
\subsection{Beginning of the proof of Theorem \ref{stabantipeakonpeakon}}\label{sec66}
Let $ \vec{c} $ and $ A>0 $ be fixed and let $ B =B(\vec{c},A) \ge 1 $ to be fixed at the end of this section.
Let $ \tilde{\alpha}_0 $ be the minimum and $ \tilde{L}_0 $ be the maximum of respectively  all the $ \alpha_0(\vec{c}) $  and all the $ L_0(\vec{c})  $ appearing in the preceding statements of Section \ref{sec-multi}. We set 
\begin{equation}\label{defeps0}
\eps_0=\min\Bigl( \frac{10^{-20}}{B \tilde{K}} \Bigl(\frac{c_1\wedge |c_{-1}|}{(1+\|\vec{c}\|_1)(N_-+N_+)}\Bigr)^2,
\; \tilde{\alpha}_0\Bigr) \; .
\end{equation}
where $  \tilde{K} $ is  the constant depending on $ \vec{c} $ that appear in Lemma 
\ref{eloignement}. 
For $ \alpha>0 $ we also set 
\begin{equation}\label{defTalpha}
T_{\alpha}= \max\Bigl(\frac{2^5}{c_1\wedge |c_{-1}|)} (9+\ln(\frac{{\mathcal A}_0}{\alpha^2})) , 0 \Bigr) 
\end{equation}
with 
$$
{\mathcal A}_0=e^{27} \frac{(\|\vec{c}\|_1 +1)^3}{(c_1\wedge |c_{-1}|)^2} (1+A) \; .
$$
For $ 0<\eps<\eps_0 $ and $ L>\tilde{L}_0 $, we set $ \alpha=B(\eps+L^{-\frac{1}{8}})$.
Since $\alpha\ge L^{-1/8}$, we have  $ \ln (1/\alpha^2) \le \ln(L^{1/4})$ and thus 
$$
 T_\alpha\le \frac{2^5}{c_1\wedge |c_{-1}|} \ln (L^{3/4}) \; ,
 $$
as soon as $ L\ge {\mathcal A}_0^4  \vee e^{36}$. Therefore we set 
\begin{equation}\label{defL0}
L_0=\max(\eps_0^{-8}, {\mathcal A}_0^4,\tilde{L}_0)\; .
\end{equation}
According to Lemma \ref{exsol}, for $ L\ge L_0 $, this ensures that the solution $ {\bf u} $ to \eqref{DP} emanating from 
$\dsp {\bf u}_0=\sum\limits_{\substack{\scriptscriptstyle j=-N_- \\ \scriptscriptstyle j \neq 0}}^{\scriptscriptstyle N_+}\varphi_{c_j}(\cdot-z^0_j)$ satisfies
$$
\big\Vert {\bf u}(t)-\sum_{j=-N_-\atop j\neq 0}^{N_+}\varphi_{c_j}(x-z_j^0-c_jt) \big\Vert _{\mathcal{H}} \le e^{-L/2} 
\le L^{-1/8}  , \quad \forall t\in [0, T_{\alpha}]\;  .
$$
On the other hand, according to the continuity with respect to initial data (see Proposition \ref{weakGWP}),  for  any $ \varepsilon>0 $ there exists $ \delta=\delta
(A,\varepsilon,c)>0 $ such that for any $ u_0\in Y $ satisfying Hypothesis \ref{hyp} and \eqref{huhu0}-\eqref{huhu}
 with $ A $ and $\delta $, it holds 
$$
\|u(t)-{\bf u}(t) \|_{\H} \le \varepsilon , \quad \forall t\in [0, T_{\alpha}]\; ,
$$
where $u\in C(\R_+; H^1(\R)) $ is the solution of the (D-P) equation emanating from $ u_0$. Gathering the two above estimates we thus infer that 
\begin{equation}\label{kjs}
\big\Vert u(t)-\sum_{j=-N_-\atop j\neq 0}^{N_+}\varphi_{c_j}(x-z_j^0-c_jt) \big\Vert _{\mathcal{H}} \le\varepsilon+ L^{-{1/8}}   , \quad \forall t\in [0, T_{\alpha}]\;  .
\end{equation}

So let $ u_0\in Y\cap H^\infty(\R) $ that satisfies Hypothesis \ref{hyp}  and \eqref{huhu0}-\eqref{huhu}
 with $ A $, $\delta $ and $ L\ge L_0$. \eqref{kjs} together with the definitions  \eqref{defeps0}-\eqref{defL0}
   and   Lemma \ref{LemmaPP1loc} then ensure that
 \begin{equation}\label{kjjjs}
\big\Vert u(t)-\sum_{j=-N_-\atop j\neq 0}^{N_+}\varphi_{c_j}(x-z_j^0-c_jt) \big\Vert _{L^\infty} <  \frac{10^{-5}}{N_-+N_+} 
(c_1\wedge |c_{-1}|)
 , \quad \forall t\in [0, T_{\alpha}], 
\end{equation}
 Applying  Lemma \ref{LemmaPPloc} with the $ z_j=z_j^0+c_j t $ we obtain the existence of the local maxima (or minima) $ \xi_j(t) $. Note that \eqref{hyg33loc}  ensures that  $ \xi_i(t)-\xi_j(t) \ge 2L/3 $ for $ i> j $
  and \eqref{hyg3loc} ensures that $ \xi_i(t) $ is the only point of  maximum (resp. point of  minimum) of $ v(t)=(4-\partial_x^2)^{-1} u(t)$ 
   on $ [\xi_i(t)-L/4, \xi_i(t)+L/4] $ for $i\in [[1,N_+]] $ (resp. $i\in [[N_{-},-1]]$). 
   
 By a continuity argument it remains to prove that  for any $ T\ge T_\alpha $, if 
 \begin{equation}\label{kkks}
u(t)\in U\left(2B(\varepsilon+L^{-{1/8}}),L/2\right)  \quad\text{on} \quad [0,T] 
\end{equation}
then there exists $ \xi_{N_-}(T) <..\xi_{-1}(T)<\xi_{1} (T)< ..\xi_{N_+} (T)$ with 
$ \xi_i(T)-\xi_j(T)\ge 2L/3 $ for $ i>j $ such that 
\begin{equation}\label{kjs2}
\big\Vert u(t)-\sum_{j=-N_-\atop j\neq 0}^{N_+}\varphi_{c_j}(x-\xi_j(T)) \big\Vert _{\mathcal{H}} \le
B(\varepsilon+ L^{-{1/8}})  , \quad \forall t\in [0, T]\;  .
\end{equation}
and $ \xi_i(T) $ is the only point of global maximum (resp. point of global minimum) of $ v(t) $ 
   on $ [\xi_i(T)-L/4, \xi_i(T)+L/4] $ for $i\in [[1,N_+]] $ (resp. $i\in [[N_{-},-1]]$).
Now it is worth noticing that \eqref{kkks} together with the definitions  \eqref{defeps0}-\eqref{defL0}
 and Proposition \ref{pro2loc} ensure that there exist $ x_{N_-}(T) <..x_{-1}(T)<x_{1} (T)< ..x_{N_+} (T)$ with 
$ x_i(T)-x_j(T)\ge 3L/4 $ for $ i>j $ such that 
\begin{equation}\label{defxi}
\big\Vert u(t)-\sum_{j=-N_-\atop j\neq 0}^{N_+}\varphi_{c_j}(x-x_j(t)) \big\Vert _{\mathcal{H}} \le
\tilde{K} B(\varepsilon+ L^{-{1/8}}) , \quad \forall t\in [0, T]\;  .
\end{equation}
and Lemma \ref{LemmaPP1loc}  together with \eqref{defeps0}-\eqref{defL0} ensure that
 $$
 \big\Vert u(t)-\sum_{j=-N_-\atop j\neq 0}^{N_+}\varphi_{c_j}(x-x_j(t)) \big\Vert_{L^\infty} \le 
  \frac{10^{-5}}{N_-+N_+} 
(c_1\wedge |c_{-1}|)
   , \quad \forall t\in [0, T] \; .
 $$
 Applying  Lemma \ref{LemmaPPloc} with the $ z_j=x_j(t)$ we obtain the existence of the local maximum (or minimum) $ \xi_j(t) $. Note that \eqref{hyg33loc} ensures that $ \xi_i(t)-\xi_j(t) \ge 2L/3 $ for $ i> j $
  and \eqref{hyg3loc} ensures that $ \xi_i(t) $ is the only point of global maximum (resp. point of global minimum) of $ v(t)=(4-\partial_x^2)^{-1} u(t)$ 
   on $ [\xi_i(t)-L/4, \xi_i(t)+L/4] $ for $i\in [[1,N_+]] $ (resp. $i\in [[N_{-},-1]]$). Moreover, \eqref{hyg2loc} and again 
 Lemma \ref{LemmaPP1loc}    prove that for $ L\ge L_0 $ large enough 
  \begin{equation}\label{defxi4}
  \big\Vert u(t)-\sum_{j=-N_-\atop j\neq 0}^{N_+}\varphi_{c_j}(x-\xi_j(t)) \big\Vert_{\mathcal{H}} \le 
  2\tilde{K} B(\varepsilon+ L^{-{1/8}})
 \end{equation}
 and
 $$
 \big\Vert u(t)-\sum_{j=-N_-\atop j\neq 0}^{N_+}\varphi_{c_j}(x-\xi_j(t)) \big\Vert_{L^\infty} \le 
  \frac{10^{-5}}{N_-+N_+} 
(c_1\wedge |c_{-1}|)
   , \quad \forall t\in [0, T] \; .
 $$
Finally, Proposition \ref{pro2loc}   together with the definition \eqref{defTalpha} of $ T_\alpha $ and   \eqref{hyg33loc}  then ensure that 
\begin{equation}\label{kj22}
u(t,\cdot)-6v(t,\cdot) \le \alpha^2= (\eps+L^{-1/8})^2 \quad \text{on} \quad [x_1(t)-8,+\infty[ \quad \forall
 t\in [T_\alpha,T] \; .
\end{equation}  


For the remaining of the proof we need the following localized versions of Lemmas \ref{Lemma P4}-\ref{Lemma P1},   where the global functional $E$ and $F$ are replaced by their localized versions $E_i$ and $F_i$.
\subsection{Localized estimates}
In the sequel we set 
\begin{equation}\label{defK}
K=\sqrt{L}/8\; .
\end{equation}
Let $ x_{-N_-}(\cdot)< ..<x_{-1}(\cdot)< x_1(\cdot)< ..<x_{N_+}(\cdot) $ be the $ N_-+N_+ $  $C^1$-functions  defined on $[0,T] $ (see \eqref{defxi}) and define the function $ \Phi_i=\Phi_i(t,x) $, $i=1,..,N_{+}$,  by 
\begin{equation}\label{defphii}
\left\{
\begin{array}{l}
\Phi_{N_+}(t)=\Psi_{N_+,\sqrt{L}/8}(t)=\Psi_{\sqrt{L}/8}(\cdot-y_{N_+}(t)) \\
\Phi_i(t)=\Psi_{i,\sqrt{L}/8}(t)-\Psi_{i+1,\sqrt{L}/8}(t)=\Psi_{\sqrt{L}/8}(\cdot-y_i(t))-\Psi_{\sqrt{L}/8}(\cdot-y_{i+1}(t)),\quad i=1,...,N_{+}-1,
\end{array}
\right.
\end{equation}
where $ \Psi_{i,K} $ and the $ y_i $'s are defined in Section \ref{Sectmonotonie} \eqref{defPsi}-\eqref{defy1}. It
is easy to check that the $ \Phi_i$'s are positive functions  and that $\displaystyle \sum_{i=1}^{N_+} \Phi_{i}\equiv
\Psi_{1,\sqrt{L}/8} $.
 Since $ L\ge L_0\ge 1 $,   \eqref{defK} and \eqref{psipsi} ensure that  $ \Phi_i $ satisfies
  for $i\in \{1,...,N_+\} $, 
 \begin{equation}
 \big\vert 1-\Phi_{i}\big\vert \le  2 e^{-\sqrt{L}} \mbox{ on } \left] y_i+\frac{L}{8},  y_{i+1}-\frac{L}{8}\right[,
  \label{de1}
 \end{equation}
 and
\begin{equation}
 \big\vert\Phi_{i}\big\vert \le 2 e^{-\sqrt{L}} \mbox{ on } \R\backslash\left]y_i-\frac{L}{8},
 y_{i+1}+\frac{L}{8}\right[ \; ,\label{de2}
 \end{equation}
 where we set $ y_{N_{+}+1}:=+\infty $. \\
 It is worth noticing that, somehow, $ \Phi_i(t) $ takes care of only the ith bump of $u(t)$.
 We will use  the following localized  version of $ E $ and $ F $ defined  for
$i\in \{1,..,N_+\}, $ by
 \begin{equation}\label{defEi}
 E_i(t) = \int_{\R}(4v^2+5v_x^2+v_{xx}^2) \Phi_{i}(t)  \mbox{ and }
 F_i(t)= \int_{\R}\left(-v^{3}_{xx}+12vv^{2}_{xx}-48v^{2}v_{xx}+64v^{3}\right) \Phi_i(t) \; .
 \end{equation}
In the statement of the four following lemmas we fix the time. This corresponds  to fix 
$x_{-N_{-}}< ..<x_{-1}<x_1<..<x_{N_+}$ with $x_i-x_j>3L/4 $ for $i>j $ such that  
\begin{equation}\label{defxixi}
\big\Vert u(t)-\sum_{j=-N_-\atop j\neq 0}^{N_+}\varphi_{c_j}(x-x_j) \big\Vert _{\mathcal{H}} \le
\tilde{K} B(\varepsilon+ L^{-{1/8}})
\end{equation}
and to fix 
$(y_1,..,y_{N_+})\in \R^{N_+} $, such that 
$$
x_{-1} +L/4<y_1<x_1<y_2<x_2<\cdot\cdot<y_{N_+} <x_{N_+}<y_{N_+ +1}=+\infty
$$
with  $ |y_i-x_j| \ge L/4 $ for $(i,j)\in [[1,N_+]]^2 $.
In particular, 
 $ E_i $ and $ F_i $ do not depend on
time.\\
For $i=1,...,N_+$, we set $\Omega_i=]y_i-L/8,y_{i+1}+L/8[$, the interval in which the mass of each peakon $\varphi_{c_i}$ (and smooth peakon $\rho_{c_i}$) is concentrated. One can see that 
\begin{equation}\label{1877}
\sum_{j=-N_-\atop j\neq 0}^{N_+}\rho_{c_j}(x-x_j)=\rho_{c_i}(x-x_i)+O(e^{-L/4}),\quad\forall x\in\Omega_i,
\end{equation}
and that  $\rho_{c_i}(x-x_i)=O(e^{-L/4})$ for all $x\in\R\backslash\Omega_i$. We will decompose $\Omega_i$ as in Section \ref{sect5} by setting
\begin{equation}\label{188}
\Theta_{i}=[x_i-6.7,x_i+6.7],\text{ where }6.7\simeq\ln\left(\frac{20}{20-\sqrt{399}}\right)>\ln\sqrt{2},\text{ with }\rho_{c_i}(\pm6.7)\simeq c_i/2400.
\end{equation}

\begin{lemma}[See \cite{AK}]\label{Lemma P444}
Let $u\in L^{2}(\mathbb{R})$ satisfying \eqref{defxixi}. Denote by $M_i=\max_{x\in\Theta_i}v(x)=v(\xi_i)$ and define for $i=1,...,N_+$ the function $g_i$ by
\begin{equation}
  g_i(x)=\left\{
    \begin{aligned}
     &2v(x)+v_{xx}(x)-3v_{x}(x),~~\forall x<\xi_i,\\
     &2v(x)+v_{xx}(x)+3v_{x}(x),~~\forall x>\xi_i.\\
    \end{aligned}
  \right.
  \label{GG11}
\end{equation}
Then it holds
\begin{equation}
\int_{\mathbb{R}}g_i^{2}(x)\Phi_i(x)dx=E_i(u)-12M_i^{2}+\Vert u\Vert_{\mathcal{H}}^2O(L^{-1/2}),
\label{GG22}
\end{equation}
and
\begin{eqnarray}\label{improvmentt}
\int_{\mathbb{R}}g_i^{2}(x)\Phi_i(x)dx&=&E_i(u-\varphi_{c_i}(\cdot-\xi_i))-12\left(\frac{c_i}{6}-M_i\right)^2+\Vert u-\varphi_{c_i}(\cdot-\xi_i)\Vert_{\mathcal{H}}^2O(L^{-1/2})\nonumber \\
& \le& O(\Vert u-\varphi_{c_i}(\cdot-\xi_i)\Vert_{\mathcal{H}}^2)
\end{eqnarray}
\begin{proof}
The proof is similar to the one of Lemma 4.3 in \cite{AK} using  \eqref{de1} and  
$\vert\Phi'\vert+\vert\Phi''\vert\leq O(L^{-1/2})$  since $K=\sqrt{L}/8$. The second identity 
follows as \eqref{improvment} in Lemma \ref{Lemma P4}.\end{proof}
\end{lemma}

\begin{lemma}[See \cite{AK}]\label{Lemma P555}
Let $u\in L^{2}(\mathbb{R})$ satisfying \eqref{defxixi}. 
 Denote by $M_i=\max_{x\in\Theta_i}v(x)=v(\xi_i)$ and define for $i=1,...,N_+$ the function $h_i$ by
\begin{equation}
  h_i(x)=\left\{
    \begin{aligned}
     &-v_{xx}-6v_{x}+16v,~~x<\xi_i,\\
     &-v_{xx}+6v_{x}+16v,~~x>\xi_i.\\
    \end{aligned}
  \right.
  \label{HH11}
\end{equation}
Then, it holds
\begin{equation}
\int_{\mathbb{R}}h_i(x)g_i^{2}(x)\Phi_i(x)dx=F_i(u)-144M_i^{3}\Phi_i(\xi_i)+\Vert u\Vert_{\mathcal{H}}^3O(L^{-1/2}).
\label{HH22}
\end{equation}
\begin{proof}
The proof is similar to the one of Lemma 4.3 in \cite{AK} using the fact that $K=\sqrt{L}/8$ and thus $\vert\Phi'\vert+\vert\Phi''\vert\leq O(L^{-1/2})$.
\end{proof}
\end{lemma}


\begin{lemma}[Connection between the conservation laws $E_i$ and $F_i$]\label{Lemma P666}
Let $u\in L^{2}(\mathbb{R})$ satisfying  Hypothesis \ref{hyp} and \eqref{defxixi}. \\
If  
\begin{equation}\label{assum6loc}
u-6v\le  \alpha^2 \quad \text{on} \quad  [x_1-8,+\infty[
\end{equation}
and for $ i\in [[1,N_+]]$,
\begin{equation}\label{assum7loc}
  \sup_{x\in\left]y_i-\frac{L}{8},
 y_{i+1}+\frac{L}{8}\right[\setminus \Theta_{\xi_i}} (|u(x)|,|v(x)|,|v_x(x)|) \le \frac{c_i}{100}\; , 
\end{equation}
then it holds
 \begin{equation}
F_i(u)\leq 18M_iE_i(u)-72M_i^{3}+O(\alpha^4)+\Vert u\Vert_{\mathcal{H}}^3O(L^{-1/2}),\quad i=1,...,N_+.
\label{198}
\end{equation}
\end{lemma}
\begin{proof}
Recall that, according to Subsection \ref{sec66},  $ v=(4-\partial_x^2)^{-1} u $ has a got a unique global maximum
$ \xi_i $ on  $\xi_i\in]y_i+L/8,y_{i+1}-L/8[$ for $ i\in [[1,N_+]]$. , it follows from \eqref{de1} that $\Phi_i(\xi_i)=1+O(e^{-\sqrt{L}})$. Combining this with $K=\sqrt{L}/8$, \eqref{GG22} and \eqref{HH22} one may deduce that
\begin{equation}
\int_{\mathbb{R}}g_i^{2}(x)\Phi_i(x)dx=E_i(u)-12M_i^{2}+\Vert u\Vert_{\mathcal{H}}^2O(L^{-1/2}),
\label{GG222}
\end{equation}
and
\begin{equation}
\int_{\mathbb{R}}h_i(x)g_i^{2}(x)\Phi_i(x)dx=F_i(u)-144M_i^{3}+\Vert u\Vert_{\mathcal{H}}^3O(L^{-1/2}).
\label{HH222}
\end{equation}
Now, in view of \eqref{de2}  and \eqref{u0uinf} it holds
$$
\left\vert\int_{\R\setminus\Omega_i}  h_i(x)g_i^{2}(x)\Phi_i(x)dx\right\vert=\|u\|_{\H} ^2 (\|u\|_{L^\infty} 
+\|u\|_{\H}) O(e^{-\sqrt{L}/8})= \|u\|_{\H}^3 O(e^{-\sqrt{L}/8})\; .
$$
It thus remains to  show that the function $ h_i $ defined in Lemma \ref{Lemma P555} satisfies $h_i\leq 18M_i+\alpha^2$ on $\Omega_i$.
 We divide $ \Omega_i $ into three intervals.  
If $x\in\Omega_i\setminus\Theta_{\xi_i}$, then using \eqref{assum6loc}, it holds
\begin{align}
h_i(x)&\le \vert u(x)\vert+6\vert v_x(x)\vert+12\vert v(x)\vert\le\frac{19c_i}{100}\le 18M_i.
\label{hles18mm}
\end{align}
If $\xi_i-6.7<x<\xi_i$, then $v_x\geq0$ and using that $u-6v\leq\alpha^2$, we get 
\begin{equation}
h_i(x)\leq18M_i+\alpha^2.
\end{equation}
If $\xi_i<x<\xi_i+6.7$, then $v_x\leq0$ and using that $u-6v\leq\alpha^2$, we get 
$$
h_i(x)\leq18M_i+\alpha^2.
$$
 Combining \eqref{GG222}, \eqref{de2}, and \eqref{HH222}, one  deduce that
\begin{align*}
F_i(u)-144M_i^3&=\int_{\R}h_i(x)g_i^2(x)\Phi_i(x)dx+\Vert u\Vert_{\mathcal{H}}^3O(L^{-1/2})\\
&=\int_{\Omega_i}h_i(x)g_i^2(x)\Phi_i(x)dx+\Vert u\Vert_{\mathcal{H}}^3O(L^{-1/2})\\
&\leq18 M_i\Big( E_i(u)-12M_i^2\Big)+ O(\alpha^4)+\Vert u\Vert_{\mathcal{H}}^3O(L^{-1/2}),
\end{align*}
that completes the proof of the lemma.
\end{proof}

\begin{lemma}\label{lemma22}
Let $u_0\in Y$ satisfying Hypothesis 1 and \eqref{huhu0}-\eqref{distz}. It holds
\begin{align}
\big\vert E_i(u_0)-E(\varphi_{c_i})\big\vert+ \big\vert F_i(u_0)-F(\varphi_{c_i})\big\vert&\leq O(\eps^4)+O(e^{-\sqrt{L}}),\qquad i\in [[-N_{-},N_+]]\setminus \{0\}.\label{locestg}
\end{align}
\end{lemma}
\begin{proof}
 It follows easily from  \eqref{huhu0}-\eqref{distz}, \eqref{defEi}, the exponential decay of $\varphi_{c_i}$ and $\Phi_i$ 
  and the choice $ K=\sqrt{L}/8 $  (see Lemma 4.7 in \cite{AK} for details).
\end{proof}

\subsection{End of the proof of Theorem \ref{stabantipeakonpeakon}}
  For $ i\in [[1,N_+]$ we set  $M_i=v\big(T,\xi_i(T)\big) $ and $\delta_i=c_i/6-M_i$. It is worth recalling that Lemma \ref{LemmaPPloc} 
 ensures that for $ 1\le i\le N_+ $, $v(T,\xi_i(T))=\max_{[y_i(T),y_{i+1}(T)]} v(T,\cdot) $, where the $ y_i$'s are defined in 
 \eqref{defyi}.  For a function $f\colon \R_{+}\longrightarrow \R$, we set
$$
\Delta_0^Tf=f(T)-f(0).
$$
Summing \eqref{198} over $i\in\{1,...,N_+\}$, we get
 $$
\sum_{i=1}^{N_+}\Delta_0^TF_i(u)\leq18\sum_{i=1}^{N_+}M_i\Delta_0^TE_i(u)+ \sum_{i=1}^{N_+}\Big[-72M_i^{3}+18M_iE_i(u_0)-F_i(u_0)\Big]\\
+O(\alpha^4)+O(L^{-1/2}) 
$$
that can be rewritten  after some computations as
 \begin{multline}
 \sum_{i=1}^{N_+}\Big[M_i^{3}-\frac{1}{4}M_iE(\varphi_{c_i})+\frac{1}{72}F(\varphi_{c_i})\Big]\leq \frac{1}{4}\sum_{i=1}^{N_+}
 \Bigl( M_i \Delta_0^TE_i(u)-\frac{1}{18}\Delta_0^TF_i(u)\Bigr) \\
+\frac{1}{4} \sum_{i=1}^{N_+}M_i\big\vert E_i(u_0)-E(\varphi_{c_i})\big\vert+\frac{1}{72} \sum_{i=1}^{N_+}\big\vert F_i(u_0)-F(\varphi_{c_i})\big\vert+O(\alpha^4)+O(L^{-1/2}).
\label{211}
\end{multline}
Using Abel transformation, the fact that $E(\varphi_{c_i})=c_i^2/3$, $F(\varphi_{c_i})=2c_i^3/3$ and  definition \eqref{defjk}, (noticing that $0\leq1/18M_1<2/3c_{1}$ ) we obtain 
 \begin{multline}
 \sum_{i=1}^{N_+}\delta_i^2\left[\frac{c_i}{2}-\delta_i\right]\leq \frac{1}{4}M_1 \Delta_0^T\mathcal{J}_{1,\frac{1}{18M_1},K}
 +\frac{1}{4} \sum_{i=2}^{N_+}\Big(M_i-M_{i-1}\Big)\Delta_0^T\mathcal{J}_{i,0,K}\\
+\frac{1}{4} \sum_{i=1}^{N_+}M_i\big\vert E_i(u_0)-E(\varphi_{c_i})\big\vert+\frac{1}{72} \sum_{i=1}^{N_+}\big\vert F_i(u_0)-F(\varphi_{c_i})\big\vert+O(\alpha^4)+O(L^{-1/2}).
\label{211'}
\end{multline}
 Now, in view of Lemma \eqref{defxi4} and \eqref{1877}
$$
M_i=\frac{c_i}{6}+O(e^{-L/4})+O(\alpha),
$$
and thus for $0<\alpha<\alpha_0(\vec{c})\ll1$  small enough and $ L>L_0\gg1$ large enough, it holds 
\begin{equation}\label{MM}
0<M_{1}<...<M_{N_+}\qquad\text{and}\qquad\delta_i<c_i/4,\qquad \text{with }i=1,...,N_+.
\end{equation}
Combining \eqref{locestg}, \eqref{211'}, \eqref{MM} and \eqref{monotonicityestim}, we obtain
\begin{equation}\label{deltai}
\sum_{i=1}^{N_+}|c_i\delta_i|\leq O(\eps^2+L^{-1/4}) \; .
\end{equation}

Now, it is again crucial to note that  (D-P) is invariant by the change of unknown $ u(t,x)\mapsto \tilde{u}(t,x)=-u(t,-x) $. 
 As in the proof of Proposition \ref{prop5} it is clear that $ \tilde{u}(0,\cdot)=-u_0(-\cdot)$ satisfies Hypothesis \ref{hyp}  with $ \tilde{x}_0=-x_0 $ and then $ \tilde{x}_0(t)=-x_0(t) $ for all $t\ge 0$.
 $\tilde{u} $ satisfies \eqref{defxi} on $ [0,T] $ with $ N_-$ and $ N_+ $ respectively replaced by
  $ \tilde{N}_-=N_+ $ and $ \tilde{N}_+=N_- $,  $ x_i(t) $ replaced by $ \tilde{x}_i(t)=-x_{-i}$ and $ c_i $ replaced by 
  $\tilde{c}_i=-c_{-i} $ . Also we notice that the definition of $ T_\alpha $ is symmetric in $ c_1$ and $-c_{-1} $ so that 
 $ \tilde{u} $ also satisfies \eqref{kj22} with $ v$ replaced by $\tilde{v} $ and $ x_1(t)$ replaced by $ \tilde{x}_1(t) $.
  Therefore, applying the above procedure  for $ \tilde{u}  $ we obtain as well that 
  \begin{equation}\label{deltaineg}
 \sum_{i=1}^{N_{-}}  |c_{-i} (c_{-i}/6-M_{-i})|=\sum_{i=1}^{\tilde{N}_+}|\tilde{c}_i (\tilde{c}_i/6-\tilde{M}_i )|\leq O(\eps^2+L^{-1/4}) \; ,
  \end{equation}
  with $ \tilde{M}_i=-M_{-i} $ where $ M_{-i}=v(T,\xi_{-i})=\min_{\Omega_i} v(T,\cdot) $. 
 
 To conclude the proof we need the following estimate on the left-hand side member of \eqref{203}.
 \begin{lemma}\label{lemma222}
For  any $u_0\in L^2(\R) $ satisfying  \eqref{huhu}-\eqref{distz}, it holds
\begin{align}
\Big\vert E(u_0)-\sum_{i=-N_{-}}^{N_+}E(\varphi_{c_i})\Big\vert&\leq O(\eps^4)+O(e^{-L/2})\label{zaq}\
\end{align}
\end{lemma}
\begin{proof}
 It follows easily from  \eqref{huhu}-\eqref{distz} and  the exponential decay of $\rho_{c_i}
 =(4-\partial_x^2)^{-1} \varphi_{c_i}$ (see Lemma 4.7 in \cite{AK} for details).
\end{proof}

Gathering \eqref{deltai}-\eqref{deltaineg}, Lemma \ref{lemma20} and \eqref{zaq} with $ \delta\le \eps^4$ we obtain that there exists $ C>0 $ only depending on 
  $ \vec{c} $ such that 
$$
\left\Vert u(T)-\sum\limits_{\substack{ i=-N_- \\  i \neq 0}}^{ N_+}\varphi_{c_i}(\cdot-\xi_i(T))\right\Vert _{\mathcal{H}}\leq   C(\eps+L^{-1/8}),
$$
and \eqref{kkks} holds by choosing $B=C\vee 1$.

\section{Appendix}
\subsection{Proof of Lemma \ref{lemvir}}
Identity \eqref{163} is a simplified version of the one derived in \cite{AK} Appendix $4.4$. We start by applying the operator $(4-\partial_x^2)^{-1}(\cdot)$ on the both sides of equation \eqref{DP} and using the fact that
\begin{equation}\label{1..32}
(4-\partial_x^2)^{-1}(1-\partial_x^2)^{-1}(\cdot)=\frac{1}{3}(1-\partial_x^2)^{-1}(\cdot)-\frac{1}{3}(4-\partial_x^2)^{-1}(\cdot),
\end{equation}
we infer that $v=(4-\partial_x^2)^{-1} u $ satisfies
\begin{equation}\label{vthx}
v_t+\frac{1}{2}h_x=0,
\end{equation}
where $ h=(1-\partial_x^2)^{-1} u^2 $. 
With this identity in hand one may check that
\begin{align*}
4\frac{d}{dt}\int_{\R}v^2gdx=8\int_{\R}vv_tgdx=-4\int_{\R}vh_xgdx \; .
\end{align*}
Since $\partial_x^2(1-\partial_x^2)^{-1}(\cdot)=-(\cdot)+(1-\partial_x^2)^{-1}(\cdot)$, \eqref{vthx} then leads to 
\begin{align*}
5\frac{d}{dt}\int_{\R}v_x^2gdx&=10\int_{\R}v_xv_{xt}gdx=-5\int_{\R}v(1-\partial_x^2)^{-1}\partial_x^2(u^2)gdx=5\int_{\R}u^2v_xgdx-5\int_{\R}v_xhgdx\\
&=5\int_{\R}u^2v_xgdx+5\int_{\R}vh_xgdx+5\int_{\R}vhg'dx \; .
\end{align*}
Moreover in the same way one may write
\begin{align*}
\frac{d}{dt}\int_{\R}v_{xx}^2gdx&=2\int_{\R}v_{xx}v_{xxt}gdx=-\int_{\R}v_{xx}(1-\partial_x^2)^{-1}\partial_x^3(u^2)gdx=\int_{\R}\partial_x(u^2)v_{xx}gdx-\int_{\R}v_{xx}h_xgdx\\
&=A_1+A_2,
\end{align*}
where since $v_{xx}=u-4v$, it holds
\begin{align*}
A_1=-\int_{\R}\partial_x(u^2)ugdx+4\int_{\R}\partial_x(u^2)vgdx=\frac{2}{3}\int_{\R}u^3g'dx-4\int_{\R}u^2v_xgdx-4\int_{\R}u^2vg'dx
\end{align*}
and 
\begin{align*}
A_2=\int_{\R}v_x(1-\partial_x^2)^{-1}\partial_x^2(u^2)gdx+\int_{\R}v_xh_xg'dx=-\int_{\R}u^2v_xgdx+\int_{\R}v_xhgdx+\int_{\R}v_xh_xg'dx.
\end{align*}
Gathering the above identities, \eqref{163} follows. We now concentrate on the proof of \eqref{164}. Using equation \eqref{DP3} one may write
\begin{equation}
\frac{d}{dt}\int_{\R}u^3gdx=-\frac{3}{2}\int_{\R}u^2(u^2)_xgdx-\frac{9}{2}\int_{\R}u^2h_xgdx=I_1+I_2.
\end{equation}
First, by integration by parts one may have
$$
I_1=\frac{3}{4}\int_{\R}u^4g'dx.
$$
Second, substituting $u^2$ by $h-h_{xx}$ and integrating by parts we get 
\begin{align*}
I_2&=-\frac{9}{2}\int_{\R}hh_xgdx+\frac{9}{2}\int_{\R}h_xh_{xx}gdx=\frac{9}{4}\int_{\R}(h^2-h_x^2)g'dx
\end{align*}
that proves \eqref{164}. Finally, \eqref{go2} can be deduced  directly from \eqref{DP2} by integrating by parts in the following way :
 \begin{eqnarray*}
 \frac{d}{dt}\int_{\R}  y g \, dx & = &-\int_{\R} \partial_x (y u) g -3 \int_{\R} y u_x g \nonumber \\
 & = & \int_{\R} y u g' -3\int_{\R} (u-u_{xx}) u_x g \nonumber \\
  &= &  \int_{\R} y u g' +\frac{3}{2} \int_{\R} (u^2-u_x^2) g'  \; . \end{eqnarray*}
\providecommand{\href}[2]{#2}

\end{document}